\documentclass[12pt]{article}
\usepackage{amssymb,amsmath,amsthm,epsfig}
\usepackage{latexsym, enumerate}
\usepackage{eepic}
\usepackage{epic}
\usepackage{graphicx}
\usepackage{color}
\usepackage{ifpdf}
\usepackage{subfigure}

\usepackage{dsfont}
\usepackage{multirow}
\usepackage{makecell}
\usepackage{algorithm}
\usepackage{bm}

\theoremstyle{plain}
\newtheorem{lem}{Lemma}[section]
\newtheorem{thm}[lem]{Theorem}

\theoremstyle{definition}

\theoremstyle{remark}
\newtheorem{rem}{Remark}[section]

\begin{document}

\title{ \large\bf A novel variable-separation method based on sparse representation for stochastic partial differential equations}

\author{
Qiuqi Li\thanks{College of Mathematics and Econometrics, Hunan University, Changsha 410082, China. Email:qiuqili@hnu.edu.cn.}
\and
Lijian Jiang\thanks{Institute of Mathematics, Hunan University, Changsha 410082, China. Email: ljjiang@hnu.edu.cn. Corresponding author}
}

\date{}
\maketitle

\begin{center}{\bf ABSTRACT}
\end{center}\smallskip

In this paper, we propose a novel variable-separation (NVS) method for generic multivariate functions. The idea of  NVS is extended to
to obtain  the solution in  tensor product structure for stochastic partial differential equations (SPDEs).
 Compared with many widely used  variation-separation methods, NVS shares their merits but has  less computation complexity and  better efficiency.
 NVS can be used to get the separated representation of the solution for SPDE in a systematic enrichment manner.
 No iteration is performed at each enrichment step. This is a significant improvement compared with proper generalized decomposition. Because  the stochastic functions of the separated representations obtained by NVS depend on the previous terms, this impacts on the computation efficiency and brings great challenge for numerical simulation for  the problems in high stochastic dimensional spaces.
 In order to overcome the difficulty, we propose an improved least angle regression algorithm (ILARS) and a hierarchical sparse low rank tensor approximation (HSLRTA) method based on sparse regularization. For ILARS, we explicitly give the selection of the optimal regularization parameters at each step based on least angle regression algorithm (LARS) for lasso problems such that  ILARS is much more efficient.
 HSLRTA hierarchically decomposes a high dimensional problem into some low dimensional problems and brings  an accurate approximation for the solution to SPDEs in high dimensional stochastic spaces using limited computer resource.
 A few numerical examples are presented to illustrate the efficacy  of the proposed methods.

\smallskip
{\bf keywords:} Novel variable-separation , Sparse regularization, Improved least angle regression algorithm, Hierarchical sparse low rank tensor approximation

\section{Introduction}

Many model inputs (e.g., model coefficients and forcing terms) often contain some uncertainties because of lacking enough knowledge about the physical properties and  measurement noise.
It is necessary to explore the uncertainty propagation for these models.  Thus uncertainty quantification is explosively growing in many branches of science and engineering.
These models involving uncertainty  can be often described by stochastic partial differential equations (SPDEs), precisely speaking, partial differential equations with random inputs.
Many numerical methods have been proposed to solve SPDEs in recent years. To predict the uncertainty propagation for the complex physical and engineering systems, spectral stochastic methods have been extensively investigated in last two decades (e.g., \cite{ghanem1999ingredients, matthies2008stochastic, nouy2009recent, beylkin2005algorithms}).
Most of these methods use a suitable set of basis functions of basic random variables, which are independent of the models.
Many numerical methods have been proposed to compute the approximate solution, such as $L^2$ projection \cite{ghiocel2012stochastic, keese2003review},  Galerkin projections \cite{babuvska2005solving, frauenfelder2005finite, matthies2005galerkin},  regression \cite{blatman2008sparse} and stochastic interpolation \cite{babuvska2007stochastic, nobile2008sparse, xiu2007efficient, xiu2005high, ganapathysubramanian2007sparse}.

Galerkin spectral stochastic methods rely on a fruitful marriage of probability theory and approximation theory in functional analysis, which yield the accurate predictions and a better control on numerical simulations through a posteriori error estimation and adaptive approximation \cite{keese2004adaptivity, mathelin2007dual, wan2005adaptive, wan2009error}. When the physical model has high dimensional random inputs, numerical simulations are generally prohibitive with the above mentioned techniques. Moreover, a good knowledge of the mathematical structure of the physical model is required to produce predictions of the behavior of the stochastic problem.  In recent years, a proper generalized decomposition (PGD) method  has been  proposed for solving SPDEs \cite{nouy2007generalized, nouy2009recent, nouy2009generalized, nouy2010proper}, which can reduce  the above mentioned limitations of Galerkin spectral stochastic methods. This method is devoted to seek the approximation of the solution with the tensor product structure under the form
\begin{equation}
\label{vsf}
u(x,\bm{\xi}):=\sum_{i=1}^{N}\zeta_i(\bm{\xi})g_i(x),
\end{equation}
which allows  a priori computation of a quasi-optimal separated representation of the solution,
where all the $g_i(x)$ are deterministic functions of the physical variables $x$ and the $\zeta_i(\bm{\xi})$
are functions of the  random variables $\bm{\xi}$. The main idea of the PGD method is devoted to constructing optimal reduced basis from a double orthogonality criterium  \cite{nouy2010proper}. The PGD method requires the solutions of a few uncoupled deterministic problems solved by classical deterministic solution techniques, and the solutions of stochastic algebraic equations solved by classical spectral stochastic methods.
However, PGD requires many iterations with the arbitrary initial guess to compute $\zeta_i(\bm{\xi})$ and $g_i(x)$ at each enrichment step $i$. This will
negatively effect on the simulation efficiency.  In this paper, we propose a novel variable-separation (NVS) method to get a separated  representation without the iteration at each enrichment step. Moreover, NVS can alleviate  the ``curse of dimensionality"  when dealing with problems in high stochastic dimension spaces.
In this work, we develop the strategy of NVS for generic multivariate function. NVS gives a variable-separation for a random field, and this can be   used to get an affine representation for model's inputs to achieve
offline-online computation decomposition, which is often desirable for uncertainty quantification of stochastic models.  Compared with classic variable-separation techniques such as  Empirical Interpolation Method (EIM),
NVS  shares the same merits as them, but it is much easier for implementation than those classic methods.
  In NVS, the optimal parameter values and interpolation nodes are not necessary. In addition, for the online computation of NVS, we can compute the approximation straightforwardly through   by the separated  representation instead of solving an algebraic system based on the optimal parameter values and interpolation nodes, which is necessary for EIM. NVS leads to very fast online computation. This is very crucial for the many-query context such as optimization, process control and inverse analysis.

Although NVS has some advantages over some classic variable-separation techniques, the stochastic function $\zeta_k(\bm{\xi})$  in (\ref{vsf}) obtained by NVS still depends on the previous functions $\{\zeta_i(\bm{\xi})\}_{i=1}^{k-1}$.
  This can effect on the computation efficiency and may bring  challenge for numerical simulation especially when the number of terms $N$  is large. To avoid this issue,
  we find a surrogate  for $\zeta_k(\bm{\xi})$ using  a suitable set of basis functions (e.g., polynomial chaos basis and radial basis functions) of the random variables.
  To this end, we propose two strategies in the paper, i.e., improved least angle regression algorithm (ILARS) and hierarchical sparse low rank tensor approximation  (HSLRTA) method, to get the approximation $\hat{\zeta}_k(\bm{\xi})$ of $\zeta_k(\bm{\xi})$ such that $\{\hat{\zeta}_i(\bm{\xi})\}_{i=1}^{N}$ are mutually independent.
It is known that the number of the effective basis functions may be small for many practical models \cite{rish2014sparse}. The optimization methods from compressive sensing  are used to extract the effective basis functions and obtain a sparse representation. There are roughly two classes of approaches to obtain the sparse representation: optimization based on $l_0$-norm and convex optimization  \cite{tropp2007signal, elad2009sparse, rish2014sparse}. The typical  methods of the $l_0$ optimization include  orthogonal matching pursuit (OMP) and iterative hard thresholding. The convex optimization based on $l_1$-norm includes least angle regression, coordinate descent and proximal methods. In this work, we  consider least angle regression (LARS) method for convex optimization based on $l_1$-norm. Julien Mairal described  the core of LARS in \cite{elad2009sparse} by seeking the solution such that a sub-gradient set contains the zero. The selection of the regularization parameter $\lambda$ is very important for  LARS method. The algorithm by Julien Mairal selected the new atoms by decreasing the value of regularization parameter $\lambda$, but the way how to decrease the value of regularization parameter is still not clear. As we know, if the step size of decreasing the value of $\lambda$ is too small, it substantially impacts on the computation efficiency. On the other hand, we may not find the non-zero coefficients exactly if the step size is too large. Therefore, it is crucial  to  select the optimal regularization parameter $\lambda$ at each step. In this work, we develop an improved least angle regression (ILARS), which explicitly gives the selection of the optimal regularization parameter $\lambda$ at each step. However, the dimension of  approximation spaces  drastically  increases with respect to dimension of random inputs, which makes it infeasible to get a good approximation of the model output $w(\bm{\xi})$ by ILARS directly when the dimension of random inputs is high because the computational cost becomes prohibitively expensive. To overcome the high dimensionality difficulty,  we propose a hierarchical sparse low rank tensor approximation, which
 is devoted to constructing an accurate approximation of random fields in a high dimensional stochastic space with  limited computation resource.

Low rank approximation methods have recently been applied to approximating functions in high dimensional tensor spaces \cite{hackbusch2012tensor, falco2012proper, grasedyck2013literature, khoromskij2012tensors}, and also have been used in several applications about uncertainty propagation \cite{nouy2007generalized, doostan2009least, nouy2010proper, jl2016, khoromskij2011tensor, matthies2008stochastic}. In the context of low rank approximation methods, the functions are approximated in  suitable low rank tensor subsets, which can give nice approximation properties for a large class of functions  in practical applications. In order to construct approximations in these tensor subsets, one usually uses least-squares methods based on sample evaluations of the function. Here, we adopt an alternative construction that involves sparse $l_1$-regularization with only a few function evaluations \cite{chevreuil2015least}. The sparse low rank tensor approximations (SLRTA) proposed in \cite{chevreuil2015least} requires a procedure of iterations at each step of sparse rank-one approximation with an initial guess. Here we extend the idea of NVS to SLRTA in order to avoid the iteration procedure at each step of sparse rank-one approximation. This provides an adaptive sparse low rank tensor approximations (ASLRTA) method, which achieves much better  efficiency than standard SLRTA. When the optimal rank $m$ is fairly large and the number of subsets of the random variables $\bm{\xi}$ is not small, ASLRTA may not give a good approximation of model output. For this situation, we introduce hierarchical sparse low rank tensor approximation (HSLRTA) method to cope with the challenge.
 The proposed HSLRTA method can provide high rank approximations for high dimensional stochastic problems.

The paper is structured as follows. In Section \ref{ssec:prelim},  we give some preliminaries and notations for the paper. Section \ref{apgd} is devoted to describing the details of NVS method. In Section \ref{regularization-ILARS}, we  introduce sparse regularization methods including  ILARS method. In Section \ref{hierarchical-low-rank-app}, we present the proposed ASLRTA and HSLRTA method based on sparse regularization methods . In section \ref{numerical examples},  a few  numerical examples are presented to illustrate the performance of all the proposed  methods. Finally, we make some conclusions and comments.

\section{Preliminaries and notations}
\label{ssec:prelim}

We consider a stochastic partial differential equation (SPDE) defined on a bounded physical domain (e.g., space or space-time domain) of the form
\begin{eqnarray}
\label{SPDE}
\begin{split}
\mathcal{L}(x,\bm{\xi})u(x,\bm{\xi})=f(x,\bm{\xi}), ~~~\forall ~x \in D, ~ \bm{\xi} \in \Omega,\\
\end{split}
\end{eqnarray}
where $\bm{\xi}:=(\xi_1, \cdots,\xi_{d})$ is a set of $d$ real-valued random variables, $\mathcal{L}(x,\bm{\xi})$ is a stochastic differential operator, $f(x,\bm{\xi})$ is the source term, and $u(x,\bm{\xi})$ is the  solution of SPDE. We introduce the associated finite dimensional probability space $(\Omega, \mathcal{B},P_{\bm{\xi}} )$, where $\Omega \in \mathbb{R}^d$ is the event space,  $\mathcal{B}$ is a $\sigma-$algebra on $\Omega$, and $P_{\bm{\xi}}$ is the probability measure. We note that the solution $u(x,\bm{\xi})$ of SPDE is a random field defined on the physical domain and takes values in a Hilbert space $\mathcal{V}$.

A weak formulation of (\ref{SPDE}) reads: find $u: \Omega \rightarrow \mathcal{V}$ such that
\begin{eqnarray}
\label{SPDE-strong-st}
\begin{split}
a\big(u(\bm{\xi}), v; \bm{\xi}\big)=b(v;\bm{\xi}), ~~~\forall ~v \in \mathcal{V},\\
\end{split}
\end{eqnarray}
where $a(\cdot; \cdot)$ and $b(\cdot;\cdot)$ are a bilinear form and linear form on $\mathcal{V}$, respectively. We denote the Hilbert space of the random variables with second order moments by $L_{P_{\bm{\xi}}}^2(\Omega)$, which is defined by
\[
L_{P_{\bm{\xi}}}^2(\Omega)=\bigg\{w: y\in \Omega \rightarrow w(y)\in \mathbb{R}; \int_\Omega w(y)^2 P_{\bm{\xi}}(dy)< \infty\bigg\}.
\]
The inner product in $L_{P_{\bm{\xi}}}^2(\Omega)$ is given by
\begin{eqnarray*}
\label{TENSOR PRODUCT}
(w,v)_{L_{P_{\bm{\xi}}}^2(\Omega)}:=\int_{\Omega}w(y)v(y)P_{\bm{\xi}}(dy),
\end{eqnarray*}
which induces the norm
\[
\|w\|_{L^2}^2=\|w\|_{L_{P_{\bm{\xi}}}^2(\Omega)}^2:=(w,w)_{L_{P_{\bm{\xi}}}^2(\Omega)}.
\]
The solution $u$ belongs to Hilbert space $L^2(\Omega; \mathcal{V})$, which can be identified with the tensor product space $\mathcal{V} \otimes L_{P_{\bm{\xi}}}^2(\Omega)$. For a simplicity of notation, we denote $L_{P_{\bm{\xi}}}^2(\Omega)$ by $\mathcal{S}$.
Then we define an inner product in  $\mathcal{V} \otimes \mathcal{S}$ by
\begin{eqnarray*}
\label{TENSOR PRODUCT}
(w,u)_{\mathcal{V} \otimes \mathcal{S}}=E[(w,u)_\mathcal{V}]:=\int_{\Omega}(w,u)_\mathcal{V} P_{\bm{\xi}}(dy).
\end{eqnarray*}
Thus the norm is defined by
\[
\|u\|_{\mathcal{V}\otimes \mathcal{S}}^2:=(u,u)_{\mathcal{V}\otimes \mathcal{S}}.
\]

We suppose that $\bm{\xi}$ can be split into $r$ mutually independent sets of random variables $\{\bm{\xi}_k\}_{k=1}^r$, where $\bm{\xi}_k$ takes values in $\Omega_k\in \mathbb{R}^{d_k}$, and $d=\sum_{k=1}^r d_k$. We denote the probability space associated with $\bm{\xi}_k$ by $(\Omega_k, \mathcal{B}_k,P_{\bm{\xi}_k} )$, where $P_{\bm{\xi}_k}$ is the probability law of $\bm{\xi}_k$. Therefore, the probability space $(\Omega, \mathcal{B},P_{\bm{\xi}} )$ associated with $\bm{\xi}$ has a product structure with $\Omega=\times_{k=1}^r\Omega_k$ and $P_{\bm{\xi}}=\otimes_{k=1}^rP_{\bm{\xi}_k}$. Consequently, the Hilbert space $\mathcal{S}$ is a tensor Hilbert space with the following tensor structure:
\[
\mathcal{S}=\mathcal{S}^1\otimes\cdots\otimes\mathcal{S}^r, ~~~\mathcal{S}^k:=L_{P_{\bm{\xi}_k}}^2(\Omega_k),~~k=1,\cdots,r.
\]
If the $d_k$ random variables $\bm{\xi}_k=(\xi_{k,1},\cdots,\xi_{k,d_k})$ are mutually independent and probability space $(\Omega_k, \mathcal{B}_k,P_{\bm{\xi}_k})$ has
itself a product structure: $\Omega_k=\times_{i=1}^{d_k}\Omega_{k,i}$ and $P_{\bm{\xi}_k}=\otimes_{i=1}^{d_k}P_{\xi_{k,i}}$, the Hilbert space $\mathcal{S}^k$ has the following tensor product structure:
\[
\mathcal{S}^k=\mathcal{S}^{k,1}\otimes\cdots\otimes\mathcal{S}^{k,d_k}, \quad \text{where} \quad  \mathcal{S}^{k,j}:=L_{P_{\xi_{k,j}}}^2(\Omega_{k,j}) \quad \text{for} \quad  j=1,\cdots,d_k.
\]

We introduce approximation spaces $\mathcal{S}_{n_k}^{k}\subset \mathcal{S}^{k}$ with orthonormal basis $\{\phi_j^k\}_{j=1}^{n_k}$ such that
\[
\mathcal{S}_{n_k}^{k}=\bigg\{v(\bm{\xi}_k)=\sum_{j=1}^{n_k}v_j\phi_j^k(\bm{\xi}_k); v_j\in \mathbb{R}\bigg\}=\{v(\bm{\xi}_k)=\bm{\phi}^k(\bm{\xi}_k)\mathbf{v}; \mathbf{v}\in \mathbb{R}^{n_k}\},
\]
where $\mathbf{v}$ denotes the vector of coefficients of $v$, and $\bm{\phi}^k=(\phi_1^k,\cdots,\phi_{n_k}^k)$ denotes the vector of basis functions.
Then the approximation space $\mathcal{S}_n\subset\mathcal{S}$ is obtained by
\[
\mathcal{S}_n=\mathcal{S}_{n_1}^1\otimes\cdots\otimes\mathcal{S}_{n_r}^r=\bigg\{v=\sum_{i\in I_n} v_i \phi_i; v_i\in \mathbb{R}\bigg\}=\bigg\{v(\bm{\xi})=\langle\bm{\phi}(\bm{\xi}),\mathbf{v}\rangle; \mathbf{v}\in \mathbb{R}^{n_1}\otimes\cdots\otimes\mathbb{R}^{n_r}\bigg\},
\]
where
\begin{eqnarray*}
&I_n=\times_{k=1}^r{1\cdots n_k},  \quad \phi_i=\big(\phi_{i_1}^1\otimes\cdots\otimes\phi_{i_r}^r\big)(\bm{\xi}_1\cdots\bm{\xi}_r)=\prod_{j=1}^r\phi_{i_j}^j(\bm{\xi}_j),\\
& \bm{\phi}(\bm{\xi})=(\bm{\phi}^1(\bm{\xi}_1),\cdots,\bm{\phi}^r(\bm{\xi}_r))\in \mathbb{R}^{n_1}\otimes\cdots\otimes\mathbb{R}^{n_r}.
\end{eqnarray*}
Here $\langle\cdot,\cdot\rangle$ is the canonical inner product in $\mathbb{R}^{n_1}\otimes\cdots\otimes\mathbb{R}^{n_r}$.

Given the approximation space $\mathcal{S}_n$, which  is sufficiently rich to approximate our quantity of interest $G(\bm{\xi})$, we attempt to provide  methods to approximate $G(\bm{\xi})$ in $\mathcal{S}_n$ for high dimensional applications  using only limited information on $G(\bm{\xi})$. As an alternative method, we will approximate high dimensional functions using hierarchical sparse low rank approximations.

\section{A novel variable-separation }
\label{apgd}

It is known that proper generalized decomposition (PGD) method \cite{nouy2010proper}   can be seen as a approach for the apriori construction of separated representations of the solution defined in tensor product spaces. Here we propose  a novel variable-separation (NVS) method to get a variable-separated representation for multivariate functions without using any iteration in each term, which is required in PGD.

\subsection{NVS for multivariable function}
\label{NVSDF}

We first develop the novel variable-separation for multivariable functions, which can be applied to obtain an affine representation for model's inputs (e.g., coefficients and source terms) to achieve offline-online computation. Given $G(x,\bm{\xi})$, we want to construct an approximation in the form
\begin{eqnarray}
\label{separat-presentation}
G(x,\bm{\xi})\approx G_N(x,\bm{\xi}):=\sum_{i=1}^{N}\zeta_i(\bm{\xi})g_i(x),
\end{eqnarray}
where $\zeta_i(\bm{\xi})$ only depends on $\bm{\xi}$ and $g_i(x)$ only depends on $x$.

Now we develop  an algorithm to obtain $\zeta_i(\bm{\xi})$ and $g_i(x)$ for each term of the right-hand side in (\ref{separat-presentation}). To this end, we  initialize a residual $r_0(x,\bm{\xi})=G(x,\bm{\xi})$. At step $k$, we obtain $\zeta_k(\bm{\xi})$ and $g_k(x)$ by taking
\begin{eqnarray}
\label{pgd-g}
g_k(x)=r_{k-1}(x,\bar{\bm{\xi}}),
\end{eqnarray}
where $\bar{\bm{\xi}} :=\arg\max_{\bm{\xi}\in \Xi}\|r_{k-1}(x,\bm{\xi})\|_{L^{\infty}(D)},$ and taking
\begin{eqnarray}
\label{pgd-z}
\zeta_k(\bm{\xi})=\frac{r_{k-1}(\bar{x},\bm{\xi})}{g_k(\bar{x})},
\end{eqnarray}
where $\bar{x}\in D $ is a point satisfying
\[
\bar{x} :=\arg\sup_{x\in D}|r_{k-1}(x,\bar{\bm{\xi}})|=\arg\sup_{x\in D}|g_k(x)|.
\]
Then we take $G_k(x,\bm{\xi})$ as follows
\begin{eqnarray}
\label{pgd-residual}
G_k(x,\bm{\xi}):=\sum_{i=1}^{k} g_i(x)\zeta_i(\bm{\xi}),
\end{eqnarray}
and update the residual $r_k(x,\bm{\xi})$ by
\begin{eqnarray}
\label{pgd-residual}
r_k(x,\bm{\xi})=G(x,\bm{\xi})-G_k(x,\bm{\xi}).
\end{eqnarray}
When  $\|r_k(x,\bm{\xi})\|_{\mathcal{V} \otimes \mathcal{S}}^2$ is small enough, we can stop the iteration procedure.

\begin{rem}
To obtain a set of $(\cdot,\cdot)_\mathcal{V}$-orthonormal functions $\{g_i(x); 1\leq i\leq N\}$, we can  apply the Gram-Schmidt process to $g_k(x)$ and $\{g_i(x); 1\leq i\leq k-1\}$ under the $(\cdot,\cdot)_\mathcal{V}$ inner product at each step $k$.
\end{rem}

We introduce the notion of n-width following Kolmogorov \cite{Buffa2012priori,   Kolmogoroff1936} for convergence analysis. Let $\mathcal{F}=\{G(x,\bm{\xi})\in \mathcal{V}: \bm{\xi}\in \Omega\}$. The Kolmogorov n-width of $\mathcal{F}$ in $\mathcal{V}$ is given by
\begin{eqnarray*}
\label{Kolmogorov width}
d_{n}(\mathcal{F}, \mathcal{V}):=\inf \{ E(\mathcal{F};Y_n): Y_n \quad  \text{is a n-dimensional subspace of $\mathcal{V}$} \},
\end{eqnarray*}
where $E(\mathcal{F};Y_n)$ is the angle between $\mathcal{F}$ and $Y_n$.
To prove the convergence rate for NVS, we use the following lemma.
\begin{lem}\cite{Buffa2012priori}
\label{Kolmogorov n-width}
Assume that the set $\mathcal{F}$ has an exponentially small Kolmogorov n-width $d_{k}(\mathcal{F}, \mathcal{V})\leq C e^{-\alpha k}$ with $\alpha > \log 2$, and $\{g_i(x); 1\leq i\leq k\}$ is a set of $(\cdot,\cdot)_\mathcal{V}$-orthonormal functions, then there exists $\beta > 0$, $C_1 > 0$ and $j\leq k$ such that
\begin{eqnarray*}
\label{expb}
\|g_j(x)\|_{\mathcal{V}} \leq C_1 e^{-\beta k}.
\end{eqnarray*}
\end{lem}

The following theorem shows that the approximation by NVS is convergent.
\begin{thm}
\label{theorem-nvs}
Suppose that the assumptions in Lemma \ref{Kolmogorov n-width} hold and   $G(x,\bm{\xi})\in H^s(D)$ for $\forall \bm{\xi}\in \Omega$, where $s>1$.
Let $r_i(x, \bm{\xi})$ be  given by (\ref{pgd-residual}) for $i=1,\cdots,k$.
Then there exist $\beta > 0$, $C > 0$ and $j\leq k$ such that
\begin{eqnarray*}
\label{proof-r3}
\|r_{j-1}\|_{\mathcal{V} \otimes \mathcal{S}} \leq C e^{-\beta k}.
\end{eqnarray*}
\end{thm}
\begin{proof}
By the maximization in the definition of $\bar{\bm{\xi}}$,  we have
\begin{eqnarray}
\label{zeta1}
\forall \bm{\xi}\in \Omega, ~~~ |\zeta_k(\bm{\xi})|\leq 1.
\end{eqnarray}
By (\ref{zeta1}) and the Lemma \ref{Kolmogorov n-width},  there exist $\beta > 0$, $C_1 > 0$ and $j\leq k$ such that
\begin{eqnarray}
\label{expb}
\|g_j(x)\|_{\mathcal{V}} \leq C_1 e^{-\beta k}.
\end{eqnarray}
Note that
\begin{eqnarray}
\label{proof-r1}
\begin{split}
\|r_{j-1}\|_{\mathcal{V} \otimes \mathcal{S}}^2=&\int_{\Omega}\|r_{j-1}(x,\bm{\xi})\|_\mathcal{V}^2 P_{\bm{\xi}}(dy)\\
\leq&|D|\int_{\Omega}\|r_{j-1}(x,\bm{\xi})\|_{L^{\infty}(D)}^2 P_{\bm{\xi}}(dy)\\
 \leq&|D|\max_{\bm{\xi}\in \Xi}\|r_{j-1}(x,\bm{\xi})\|_{L^{\infty}(D)}^2\\
 =&C_2 \|r_{j-1}(x,\bm{\xi}_j)\|_{L^{\infty}(D)}^2 ~~~(\text{ take  } C_2=|D|~)\\
 =&C_2 \|g_j(x)\|_{L^{\infty}(D)}^2.
\end{split}
\end{eqnarray}
Due to the Sobolev imbedding of $\mathcal{V}=H^s(D)$ ($s>1$) into $L^{\infty}(D)$, we have
\begin{eqnarray}
\label{proof-r2}
 \|g_j(x)\|_{L^{\infty}(D)} \leq C_3 \|g_j(x)\|_{\mathcal{V}}.
\end{eqnarray}
By (\ref{expb}), (\ref{proof-r1}) and (\ref{proof-r2}),
there exist $\beta > 0$ and $C =\sqrt{|D|}C_1C_3> 0$ such that
\begin{eqnarray*}
\label{proof-r3}
\|r_{j-1}\|_{\mathcal{V} \otimes \mathcal{S}} \leq C e^{-\beta k}, ~~~~(j\leq k),
\end{eqnarray*}
where $C$ depends only on $s$ and $D$. This completes the proof.
\end{proof}

The convergence analysis is based on the steps from equation (\ref{pgd-g}) to  (\ref{pgd-residual}), which imply a practical algorithm for NVS.   In practical computation, we can relax the condition to choose $\bar{\bm{\xi}}$  and $\bar{x}$ in those steps. We can randomly choose $\bar{\bm{\xi}}$ such that $\bar{\bm{\xi}}$ in each step $i$ ($i=1,\cdots N$) is only different from previous steps. For $\bar{x}$, we just want $g_k(\bar{x})\neq 0$ in equation (\ref{pgd-z}). We describe the practical  algorithm  of NVS for multivariate functions  in Algorithm \ref{algorithm-mf}.

\begin{algorithm}
\caption{NVS for multivariate functions }
     \textbf{Input}:  The  function $G(x,\bm{\xi})$, a set of samples $\Xi\in\Omega$, the error tolerance $\varepsilon$ \\
     \textbf{Output}:  The separated representation $G_N(x,\bm{\xi}):= \sum_{i=1}^N g_i(x)\zeta_i(\bm{\xi})$\\
      ~1:~~Initialize the residual $\textbf{r}_0=G(x,\bm{\xi})$, the iteration counter $k=1$; \\
      ~2:~~Take $\bar{\bm{\xi}}=\arg\max_{\bm{\xi}\in \Xi}\|r_{k-1}(x,\bm{\xi})\|_{\mathcal{V}}$ and $\bar{x}$ such that $g_k(\bar{x})\neq 0$, calculate  $\zeta_k(\bm{\xi})=\frac{r_{k-1}(\bar{x},\bm{\xi})}{g_k(\bar{x})}$\\
      $~~~~~$ and $g_k(x)=r_{k-1}(x,\bar{\bm{\xi}})$;\\
      ~3:~~Update $\Xi$ with $\Xi=\Xi\setminus \bar{\bm{\xi}}$, and take the approximation $G_k(x,\bm{\xi}):= \sum_{i=1}^k g_i(x)\zeta_i(\bm{\xi})$ and \\
      $~~~~~$the residual $\textbf{r}_k=G(x,\bm{\xi})-G_k(x,\bm{\xi})$; \\
      ~4:~~$k\rightarrow k+1$;\\
      ~5:~~Return to Step 2 if $\|r_i\|_{\mathcal{V} \otimes \mathcal{S}}^2 \geq \varepsilon$, otherwise \textbf{terminate} .\\
      ~6:~~$N=k$;
   \label{algorithm-mf}
\end{algorithm}
\begin{rem}
 When we apply Algorithm \ref{algorithm-mf} to practical computation, we
 usually take a small  sample set $\Xi$  scattered in the random space.
\end{rem}

\begin{rem}
 To get an affine representation for model's inputs, NVS has the same merits as Empirical Interpolation Method (EIM). But NVS is is much more efficient than EIM.
  This is because that: (1) In NVS, there is no need to choose the suitable parameter values and interpolation nodes based on a large training set; (2) for each $\bm{\xi} \in \Omega$, we can calculate the approximation directly by (\ref{separat-presentation}) instead of solving an algebraic system, which is required in EIM.
\end{rem}

\subsection{NVS for stochastic partial differential equations}

The NVS can be applied to SPDE and give a variable-separation representation for solution.
Let  the SPDE have the weak formulation (\ref{SPDE-strong-st}), where we  assume that the  bilinear form $a(\cdot, \cdot; \bm{\xi})$ and the associated  linear form $b(\cdot; \bm{\xi})$ are affine with respect to $\bm{\xi}$, i.e.,
\begin{eqnarray}
\label{affinely-ag}
\begin{cases}
\begin{split}
a(w,v;\bm{\xi})&=\sum_{i=1}^{m_{a}}k^{i}(\bm{\xi})a^{i}(w,v), \quad \forall w,v\in \mathcal{V},\quad \forall \bm{\xi} \in \Omega, \\
b(v;\bm{\xi})&=\sum_{i=1}^{m_{b}}f^{i}(\bm{\xi})b^{i}(v),  \quad \forall v\in \mathcal{V},\quad \forall \bm{\xi} \in \Omega.
\end{split}
\end{cases}
\end{eqnarray}
In the above, for $i= 1, \cdots , m_{a}$, each $k^{i}(\bm{\xi})\in\mathcal{S}$ is a stochastic function and each $a^{i}: \mathcal{V} \times \mathcal{V}\longrightarrow \mathbb{R}$ is  a symmetric bilinear form independent of $\bm{\xi}$. For $i= 1, \cdots, m_{b}$, each  $f^{i}(\bm{\xi})\in\mathcal{S}$ is a stochastic function and each  $b^{i}:\mathcal{V} \longrightarrow \mathbb{R}$ is continuous functional independent of $\bm{\xi}$. When  $a(\cdot, \cdot; \bm{\xi})$ and $b(\cdot; \bm{\xi})$ are not affine with respect to $\bm{\xi}$, such an affine expansion
can be obtained by using NVS presented in Section \ref{NVSDF}.

Let $\mathcal{V}_h\subset \mathcal{V}$ be a given finite dimensional approximation space. We find  the numerical solution to problem (\ref{SPDE-strong-st}) under the form
\begin{eqnarray}
\label{separat-spde}
u(x,\bm{\xi})\approx u_N(x,\bm{\xi}):=\sum_{i=1}^{N}\zeta_i(\bm{\xi})g_i(x),
\end{eqnarray}
where $\zeta_i(\bm{\xi})\in \mathcal{S}$ are stochastic functions and $g_i(x)\in\mathcal{V}_h$ are
deterministic functions. Let the residual for NVS
\[
e(\bm{\xi}): = u(\bm{\xi})-u_{k-1} (\bm{\xi}).
\]
By equation (\ref{SPDE-strong-st}), we get
\begin{eqnarray*}
\label{deduce equation}
\begin{split}
a\big(u(\bm{\xi})-u_{k-1}(\bm{\xi})+u_{k-1}(\bm{\xi}),v;\bm{\xi}\big)=b(v;\bm{\xi}), \quad \forall v\in \mathcal{V}_h,
\end{split}
\end{eqnarray*}
that is,
\[
a\big(e(\bm{\xi}),v;\bm{\xi}\big)=b(v;\bm{\xi})-a\big(u_{k-1}(\bm{\xi}),v;\bm{\xi}\big), \quad \forall v\in \mathcal{V}_h.
\]
Let $r(v;\bm{\xi})\in \mathcal{V}_h^{*}$ (the dual space of $ \mathcal{V}_h$) be the residual
\begin{eqnarray}
\label{Rfa-weak-eq}
r(v;\bm{\xi}):&=
\begin{cases}
\begin{split}
&b(v;\bm{\xi}) ,& ~k=1, \\
&b(v;\bm{\xi})-a\big(u_{k-1}(\bm{\xi}),v;\bm{\xi}\big),& ~k\geq 2.
\end{split}
\end{cases}
\end{eqnarray}
Then we get
\begin{eqnarray}
\label{SPDE-strong-st-red}
\begin{split}
a\big(e(\bm{\xi}),v;\bm{\xi}\big)=r(v;\bm{\xi}), \quad \forall v\in \mathcal{V}_h.
\end{split}
\end{eqnarray}
By Riesz representation theory,  there exists a function $\hat{e}(\bm{\xi})\in \mathcal{V}_h$ such that
\begin{eqnarray}
\label{Riesz-weak-eq}
\begin{split}
\big(\hat{e}(\bm{\xi}),v\big)_{\mathcal{V}}=r(v;\bm{\xi}),  \quad \forall v\in \mathcal{V}_h.
\end{split}
\end{eqnarray}
Then we can rewrite the error residual equation (\ref{SPDE-strong-st-red}) as
\begin{eqnarray*}
\label{Riesz error residual equation}
\begin{split}
a\big(e(\bm{\xi}),v;\bm{\xi}\big)=\big(\hat{e}(\bm{\xi}),v\big)_{\mathcal{V}}, \quad \forall v\in \mathcal{V}_h.
\end{split}
\end{eqnarray*}
Consequently,  the dual norm of the residual  $r(v; \bm{\xi})$  can be evaluated through the Riesz representation,
\begin{eqnarray}
\label{dual-norm-residual}
\begin{split}
\|r(v;\bm{\xi})\|_{\mathcal{V}^{*}}:=\sup_{v\in \mathcal{V}_h}\frac{r(v;\bm{\xi})}{\|v\|_{\mathcal{V}}}=\|\hat{e}(\bm{\xi})\|_{\mathcal{V}}.
\end{split}
\end{eqnarray}
The computation of the residual is crucial to  NVS. To efficiently compute $\|\hat{e}(\bm{\xi})\|_\mathcal{V}$, we apply an offline-online procedure presented in \cite{prv02, rhp08}.

By (\ref{Rfa-weak-eq}) and (\ref{affinely-ag}),  the residual can be expressed by
\begin{eqnarray}
\label{Qresidual-eq-right}
\begin{split}
r(v;\bm{\xi})&=b(v;\bm{\xi})-a(u_{k-1}(\bm{\xi}),v;\bm{\xi})\\
&=b(v;\bm{\xi})-\sum_{i=1}^{k-1}\zeta_i(\bm{\xi})a(g_{i},v;\bm{\xi})\\
&=\sum_{q=1}^{m_{b}}f^{q}(\bm{\xi})b^{q}(v)-\sum_{i=1}^{k-1}\zeta_i(\bm{\xi})\sum_{p=1}^{m_{a}}k^{p}(\bm{\xi})a^{p}(g_{i},v).
\end{split}
\end{eqnarray}
By (\ref{Qresidual-eq-right}) and (\ref{Riesz-weak-eq}), we have
\begin{eqnarray*}
\label{Riesz Qerror estimators}
 (\hat{e}(\bm{\xi}),v)_\mathcal{V}=\sum_{q=1}^{m_{b}}f^{q}(\bm{\xi})b^{q}(v)-\sum_{i=1}^{k-1}\zeta_i(\bm{\xi})\sum_{p=1}^{m_{a}}k^{p}(\bm{\xi})a^{p}(g_{i},v).
\end{eqnarray*}
This implies that
\begin{eqnarray}
\label{Riesz-error-estimators}
\hat{e}(\bm{\xi})=\sum_{q=1}^{m_{b}}f^{q}(\bm{\xi})\mathcal{C}_{q}+\sum_{i=1}^{k-1}\zeta_i(\bm{\xi})\sum_{p=1}^{m_{a}}k^{p}(\bm{\xi})\mathcal{L}_{i}^{p},
\end{eqnarray}
where $\mathcal{C}_{q}$ is the Riesz representation of $l^{q}$, i.e., $(\mathcal{C}_{q},v)_\mathcal{V}=b^{q}(v)$ for any $v\in \mathcal{V}$,  $1\leq q\leq m_{b}$. Similarly, $\mathcal{L}_{i}^{p}$  is the Riesz representation of $a^{p}(g_i,v)$, i.e., $(\mathcal{L}_{i}^{p},v)_\mathcal{V}=-a^{p}(g_i,v)$ for any $v\in \mathcal{V}$, where $1\leq i\leq k-1$ and $1\leq p\leq m_{a}$.
 The equation (\ref{Riesz-error-estimators}) gives rise to
\begin{eqnarray}
\label{error-norm-est}
\begin{split}
\|\hat{e}(\bm{\xi})\|_\mathcal{V}^2=&\sum_{q=1}^{m_{b}}\sum_{q'=1}^{m_{b}}f^{q}(\bm{\xi})f^{q'}(\bm{\xi})(\mathcal{C}_{q},\mathcal{C}_{q'})_\mathcal{V}+\sum_{i=1}^{k-1}\sum_{p=1}^{m_{a}}\zeta_i(\bm{\xi})k^{p}(\bm{\xi})\\
&\times\{2\sum_{q=1}^{m_{b}}f^{q}(\bm{\xi})(\mathcal{C}_{q},\mathcal{L}_{i}^{p})_\mathcal{V}+\sum_{i'=1}^{k-1}\sum_{p'=1}^{m_{a}}\zeta_{i'}(\bm{\xi})k^{p'}(\bm{\xi})
(\mathcal{L}_{i}^{p},\mathcal{L}_{i'}^{p'})_\mathcal{V}\}.
\end{split}
\end{eqnarray}

In the offline stage we compute $\mathcal{C}_{q}$ and $\mathcal{L}_{i}^{p}$, where $1\leq i\leq k-1$, $1\leq q\leq m_{b}$ and $1\leq p\leq m_{a}$. We store
$(\mathcal{C}_{q}, \mathcal{C}_{q'})_\mathcal{V}$,  $(\mathcal{C}_{q},\mathcal{L}_{i}^{p})_\mathcal{V}$, $(\mathcal{L}_{i}^{p},\mathcal{L}_{i'}^{p'})$ for online stage, where $1\leq i,i'\leq k-1$, $1\leq q,q'\leq m_{b}$,  $1\leq p,p'\leq m_{a}$.
In the online stage,  we  evaluate  $\|\hat{e}(\bm{\xi})\|_\mathcal{V}$ for any $\bm{\xi}$   using (\ref{error-norm-est}).

At step $k$, we choose
\begin{eqnarray*}
\bm{\xi}_k :&=
\begin{cases}
\begin{split}
&\text{chosen randomly in }  \Omega ,& ~k=1, \\
&\arg\max_{\bm{\xi}\in \Xi}\|\hat{e}(\bm{\xi})\|_{\mathcal{V}},& ~k\geq 2.
\end{split}
\end{cases}
\end{eqnarray*}
Let $e_h(x)$ be  equation (\ref{SPDE-strong-st-red}) with $\bm{\xi}=\bm{\xi}_k$.    We take $g_k(x)=e_h(x)$ in (\ref{separat-spde}).
Let $e(\bm{\xi}):=e_h(x)e_{\xi}(\bm{\xi})$. By equation (\ref{affinely-ag}), we get
\begin{eqnarray}
\label{Qresidual-eq-left}
\begin{split}
a(e(\bm{\xi}),v;\bm{\xi})&=e_{\xi}(\bm{\xi})\sum_{p=1}^{m_{a}}k^{p}(\bm{\xi})a^{p}(e_h(x),v),
\end{split}
\end{eqnarray}
By (\ref{SPDE-strong-st-red}), (\ref{Qresidual-eq-left}) and (\ref{Qresidual-eq-right}), we have
\begin{eqnarray}
\label{SPDE-strong-st-red-af}
e_{\xi}(\bm{\xi})\sum_{p=1}^{m_{a}}k^{p}(\bm{\xi})a^{p}(e_h(x),v)=\sum_{q=1}^{m_{b}}f^{q}(\bm{\xi})b^{q}(v)-\sum_{i=1}^{k-1}\zeta_i(\bm{\xi})\sum_{p=1}^{m_{a}}k^{p}(\bm{\xi})a^{p}(g_{i},v), ~~\forall ~v \in \mathcal{V}_h.
\end{eqnarray}
We take $v=e_h(x)$ in equation (\ref{SPDE-strong-st-red-af}), then it follows  that
\begin{eqnarray}
\label{exi}
\zeta_k(\bm{\xi}):=e_{\xi}(\bm{\xi})=\frac{\sum_{q=1}^{m_{b}}f^{q}(\bm{\xi})b^{q}(e_h(x))-\sum_{i=1}^{k-1}\zeta_i(\bm{\xi})\sum_{p=1}^{m_{a}}k^{p}(\bm{\xi})a^{p}(g_{i},e_h(x))}{\sum_{p=1}^{m_{a}}k^{p}(\bm{\xi})a^{p}(e_h(x),e_h(x))}.
\end{eqnarray}

Algorithm \ref{algorithm-pgddf} describes the procedure  for NVS to solve stochastic partial differential equations. For practical simulation, we can take a small sample set $\Xi$ in Algorithm \ref{algorithm-pgddf}.
\begin{algorithm}
\caption{NVS for stochastic partial differential equation}
     \textbf{Input}:  The stochastic differential operator $\mathcal{L}(x,\bm{\xi})$, the source term $f(x,\bm{\xi})$, a set of samples $\Xi\in\Omega$, and the error tolerance $\varepsilon$\\
     \textbf{Output}: The separated representation $u_N(x,\bm{\xi}):= \sum_{i=1}^N g_i(x)\zeta_i(\bm{\xi})$\\
      ~1:~~Initialize the residual $r(v;\bm{\xi}):=b(v;\bm{\xi})$,  a random $\bm{\xi}_1\in \Omega$,\\
      $~~~~~$the iteration counter $k=1$; \\
      ~2:~~Calculate $g_k(x)=e_h(x)$ by solving (\ref{SPDE-strong-st-red}) with $\bm{\xi}=\bm{\xi}_k$, and $\zeta_k(\bm{\xi})=e_{\xi}(\bm{\xi})$ by (\ref{exi});\\
      ~3:~~Update $\Xi$ with $\Xi=\Xi\setminus \bm{\xi}_k$, and take the approximation $u_k(x,\bm{\xi}):= \sum_{i=1}^k g_i(x)\zeta_i(\bm{\xi})$; \\
      ~4:~~Take the residual $r(v;\bm{\xi}):=b(v;\bm{\xi})-a(u_{k-1}(\bm{\xi}),v;\bm{\xi})$, $\bm{\xi}_k=\arg\max_{\bm{\xi}\in \Xi}\|\hat{e}(\bm{\xi})\|_{\mathcal{V}}$; \\
      ~5:~~$k\rightarrow k+1$;\\
      ~6:~~ return to Step 2 if $\|\hat{e}(\bm{\xi}_k)\|_{\mathcal{V}} \geq \varepsilon$, otherwise \textbf{terminate} .\\
      ~7:~~$N=k$;
   \label{algorithm-pgddf}
\end{algorithm}

By (\ref{exi}), we find that $\zeta_k(\bm{\xi})$ depends on $\{\zeta_i(\bm{\xi})\}_{i=1}^{k-1}$ computed previously, which effects on the computation efficiency leads to
great challenge for numerical simulation  when the number of terms $N$ for (\ref{separat-spde}) is  great. To overcome the difficulty,  we will propose    improved least angle regression algorithm (ILARS) and hierarchical sparse low rank tensor approximation method (HSLRTA) to construct the surrogates
  $\{\hat{\zeta}_i(\bm{\xi})\}_{i=1}^{N}$ for $\{\zeta_i(\bm{\xi})\}_{i=1}^{N}$, where $\{\hat{\zeta}_i(\bm{\xi})\}_{i=1}^{N}$ are  independent of each other.

\section{Sparse regularization and ILARS for lasso problems}
\label{regularization-ILARS}

In this section, we will present  ILARS method, which can be used to approximate each $\zeta_i(\bm{\xi})$ in (\ref{exi}) under the form
\[
\zeta_i(\bm{\xi})\approx \sum_{i=1}^{\mathbf{N}}\mathbf{v}(i)\phi_i(\bm{\xi}),
\]
for $i=1,\cdots,N$, where $\{\phi_1(\bm{\xi}),\cdots,\phi_\mathbf{N}(\bm{\xi})\}$ is a set of basis functions for the given approximation space $\mathcal{S}_\mathbf{N}$. For the presentation, we consider a real-valued model output $w:\Omega \longrightarrow \mathbb{R}$. Let $\{\bm{\xi}^{(j)}\}_{j=1}^M$ be a set of $M$ samples of $\bm{\xi}$.
For the construction of an approximation $w_\mathbf{N}\in\mathcal{S}_\mathbf{N}$, we can use the ordinary least-squares method and solve the following optimization problem:
\begin{eqnarray*}
\label{least-squares method}
\begin{split}
\|w-w_\mathbf{N}\|_{L^2}^2=\min_{v\in \mathcal{S}_\mathbf{N}}\|w-v\|_{L^2}^2.
\end{split}
\end{eqnarray*}
The ordinary least-squares method may not give good results because the solution is very sensitive to samples. In the least-square method,  it is required that the number of parameter sample scales quadratically with the number of unknowns, i.e., $M=O(\mathbf{N}^2)$, ref. \cite{Migliorati2014Analysis}. In order to avoid these issues, we will impose some sparse regularization on the
optimization problem.

\subsection{Sparse regularization}
A sparse representation  is the one that admits an accurate  approximation with only a few nonzero terms. If a stochastic function is sparse with respect to a particular  basis, e.g., polynomial chaos, we can apply sparse regularization methods to get a sparse representation  with only a few samples.
To this end,  we consider a regularized least-squares functional defined by
\begin{eqnarray*}
\label{least-squares method}
\begin{split}
\mathcal{J}(v)=\|w-v\|_{L^2}^2+\lambda\mathcal{L}(v),
\end{split}
\end{eqnarray*}
where $\mathcal{L}$ is a regularization functional, and $\lambda$ denotes the regularization parameter.
Then the solution to the regularized least-squares problem solves the optimization problem, i.e.,
\begin{eqnarray}
\label{regularized least-squares method}
\begin{split}
\mathcal{J}(w_\mathbf{N})=\min_{v\in \mathcal{S}_\mathbf{N}}\mathcal{J}(v).
\end{split}
\end{eqnarray}
We denote the coefficients of an element $v =\sum_{i=1}^\mathbf{N} v_i\phi_i(\bm{\xi}) \in \mathcal{S}_\mathbf{N}$ by $\mathbf{v}=(v_1,...,v_\mathbf{N} )^T \in \mathbb{R}^\mathbf{N}$.
Let $\mathbf{z}=(z_1,...,z_M )^T \in \mathbb{R}^M$ be the vector of the evaluations for  $\{w(\bm{\xi}^{(i)})\}_{i=1}^M$ and $\Phi=(\varPhi_1,\cdots,\varPhi_\mathbf{N})\in\mathbb{R}^{M\times \mathbf{N}}$ the matrix with components $(\Phi)_{i,j}=\phi_j(\bm{\xi}^{(i)})$. The ordinary least-squares method can be written as
 \begin{eqnarray}
\label{least-squares matrix}
\begin{split}
\mathbf{w}=\arg\min_{\mathbf{v}\in \mathbb{R}^\mathbf{N}}\|\mathbf{z}-\Phi \mathbf{v}\|_2^2,
\end{split}
\end{eqnarray}
 and the algebraic version of regularized least-squares problem can be written as follows:
 \begin{eqnarray}
\label{regularized least-squares matrix}
\begin{split}
J(\mathbf{w})=\min_{\mathbf{v}\in \mathbb{R}^\mathbf{N}}J(\mathbf{v}), ~~\text{where}~~J(\mathbf{v})=\|\mathbf{z}-\Phi \mathbf{v}\|_2^2+\lambda L(\mathbf{v}),
\end{split}
\end{eqnarray}
 where $L(\mathbf{v})$ is a function corresponding to $\mathcal{L}(v)$.

Additional information such as smoothness and  sparsity can be provided through regularization. We can obtain some special solutions by solving problem (\ref{regularized least-squares method}) with some assumptions on the regularization function.  The choice of regularization parameter $\lambda$ is crucial for solving (\ref{regularized least-squares matrix}). In this paper, we use cross validation to select an optimal value of $\lambda$.

Suppose that an approximation $\sum_{i=1}^\mathbf{N} w_i\phi_i(\bm{\xi})$ of a function $w(\bm{\xi})$ is sparse with respect to the basis $\{\phi_j\}_{j=1}^\mathbf{N}$. By sparse regularization, we can find the sparse approximation using  $M$ ($M \ll \mathbf{N}$) realizations of $w(\bm{\xi})$.
An optimal $m-$sparse approximation of $w$ can be obtained by solving the constrained optimization problem
\begin{eqnarray}
\label{sparse-optimization-0}
\begin{split}
\min_{\mathbf{v}\in \mathbb{R}^\mathbf{N}}\|\mathbf{z}-\Phi \mathbf{v}\|_2^2 ~~\text{subject to}~~\|\mathbf{v}\|_0\leq m,
\end{split}
\end{eqnarray}
 where
 \[
 \|\mathbf{v}\|_0=\sharp\{i\in\{1,\cdots,\mathbf{N}\}: v_i\neq0\}
 \]
gives the number of nonzero components of $\mathbf{v}$. In general, the optimization problem (\ref{sparse-optimization-0}) is an NP-hard problem. With the so-called restricted isometry property (RIP), (\ref{sparse-optimization-0}) can be approximated by the following convex optimization problem:
\begin{eqnarray}
\label{sparse-optimization-1}
\begin{split}
\min_{\mathbf{v}\in \mathbb{R}^\mathbf{N}}\|\mathbf{z}-\Phi \mathbf{v}\|_2^2 ~~\text{subject to}~~\|\mathbf{v}\|_1\leq \delta,
\end{split}
\end{eqnarray}
where $\|\mathbf{v}\|_1$ is the $l_1-$norm of $\mathbf{v}$. Since the convexity of $l_1-$norm,  we can consider the equivalent optimization problem of (\ref{sparse-optimization-1}), known as Lasso problem:
\begin{eqnarray}
\label{lasso}
\begin{split}
\min_{\mathbf{v}\in \mathbb{R}^\mathbf{N}}\lambda\|\mathbf{v}\|_1+\frac{1}{2}\|\mathbf{z}-\Phi \mathbf{v}\|_2^2,
\end{split}
\end{eqnarray}
where $\lambda$ corresponds to Lagrange multiplier and is related to $\delta$. There are several optimization algorithms for solving (\ref{lasso}). In this paper, we introduce an improved least angle regression algorithm (ILARS) based on sub-gradient for the lasso problem.

\subsection{ILARS for lasso problem}
\label{mlars}

Let $L(\mathbf{v})=\|\mathbf{v}\|_1$. By equation (\ref{regularized least-squares matrix}), we have
\begin{eqnarray}
\label{lasso-eq}
\begin{split}
J(\mathbf{v})=\lambda\|\mathbf{v}\|_1+\frac{1}{2}\|\mathbf{z}-\Phi \mathbf{v}\|_2^2.
\end{split}
\end{eqnarray}
The sub-gradient set is given by the set of all vectors
\begin{eqnarray}
\label{sub-gradient}
\begin{split}
\partial J(\mathbf{v}):=\lambda\mathbf{s}+\Phi^T(\Phi \mathbf{v}-\mathbf{z}),
\end{split}
\end{eqnarray}
where
\begin{eqnarray*}
\mathbf{s}(i)&=
\begin{cases}
\begin{split}
+1,& ~\mathbf{v}(i)> 0, \\
[-1,+1],&~ \mathbf{v}(i)=0, ~~i=1,\cdots,\mathbf{N},\\
-1,& ~\mathbf{v}(i)< 0.
\end{split}
\end{cases}
\end{eqnarray*}
We want to seek $\mathbf{v}$ and $\mathbf{s}$ such that $0 \in \partial J(\mathbf{v})$.
Combining with the sub-gradient requirement, we now introduce an improved  least angle regression method, which is described in Algorithm \ref{algorithm-ILARS}.

\begin{thm}
\label{theorem-lambda}
In Algorithm \ref{algorithm-ILARS}, if we have gotten the vector $\mathbf{s}$, the solution support $ S$, and the sparse solution $\mathbf{v}$ at the $(k-1)_{\text{th}}$ iteration, then the regularization parameter at the $k_{\text{th}}$ iteration should be taken such that
\[
\lambda=\|(\Phi^T(\mathbf{z}-\Phi\mathbf{v}))_{S^{c}}\|_\infty,
\]
where $S^{c}$ is the complementary set of $S$.
\end{thm}
\begin{proof}
By equation (\ref{sub-gradient}) and the requirement $0 \in \partial J(\mathbf{v})$, we get
\begin{eqnarray*}
\label{lambda1}
\begin{split}
\lambda\mathbf{s}+\Phi^T(\Phi \mathbf{v}-\mathbf{z})=0.
\end{split}
\end{eqnarray*}
Then
\[
\mathbf{s}=\frac{\Phi^T(\mathbf{z}-\Phi\mathbf{v})}{\lambda}.
\]
 In order to make $\mathbf{s}$ satisfy the condition  in equation (\ref{sub-gradient}), i.e.,  all the entries in $\mathbf{s}$ are in the range $[ -1, 1]$, we should take  $\lambda$ such that
\[
\lambda=\min\bigg\{\hat{\lambda}: -1\leq\mathbf{s}(i)=\frac{\varPhi_i^T(\mathbf{z}-\Phi\mathbf{v})}{\hat{\lambda}}\leq 1, i=1,\cdots,N\bigg\}.
\]
Let $\mathbf{I}\in \mathbb{R}^{\mathbf{N}\times1}$ be a vector with all the entries being $1$. We note that $\mathbf{s}_S=\text{sign}(\mathbf{v}_S)$, i.e., $|\mathbf{s}_S|=\mathbf{I}_S$ is independent of $\lambda$ on the support $S$.
 Thus we have
 \[
 \lambda=\|(\Phi^T(\mathbf{z}-\Phi\mathbf{v}))_{S^{c}}\|_\infty.
 \]
 Let
 \[
 \mathbf{E}_{\lambda}:=\bigg\{\hat{\lambda}: -1\leq\mathbf{s}(i)=\frac{\varPhi_i^T(\mathbf{z}-\Phi\mathbf{v})}{\hat{\lambda}}\leq 1, i=1,\cdots,N\bigg\}.
 \]
  It is obvious that $\lambda=\|(\Phi^T(\mathbf{z}-\Phi\mathbf{v}))_{S^{c}}\|_\infty\in \mathbf{E}_{\lambda}$. It remains to prove   that $\lambda=\|(\Phi^T(\mathbf{z}-\Phi\mathbf{v}))_{S^{c}}\|_\infty$ is the smallest one in the set  $\mathbf{E}_{\lambda}$.  Suppose  that there exists a $\lambda_0\in \mathbf{E}_{\lambda}$ such that $\lambda_0 \leq \lambda$. Let $j_0=\arg\max_{j}\{|\varPhi_j^T(\mathbf{z}-\Phi \mathbf{v})|, j\in S^{c}\}$. Thus $\mathbf{s}(j_0)=\frac{\|(\Phi^T(\mathbf{z}-\Phi\mathbf{v}))_{S^{c}}\|_\infty}{\lambda_0}=\frac{\lambda}{\lambda_0}>1$. This contradicts with the fact  that $\mathbf{s}$ are in the range $[ -1, 1]$. So we conclude that $\lambda=\|(\Phi^T(\mathbf{z}-\Phi\mathbf{v}))_{S^{c}}\|_\infty$.
\end{proof}

\begin{algorithm}
\caption{Improved least angle regression for lasso problem}
     \textbf{Input}: A matrix $\Phi$, the vector $\mathbf{z}$ and the given regularization parameter $\lambda_0$\\
     \textbf{Output}: The sparse solution $\mathbf{v}$, the solution support $ S$, and the approximation \\ $w_{\mathbf{N}}=\sum_{i=1}^{\mathbf{N}}\mathbf{v}(i)\phi_i(\bm{\xi})$ \\
      ~1:~~Initialize the solution $\mathbf{v}=0$, $\lambda=\|\Phi^T(\mathbf{z})\|_\infty$, and the solution support $S={\O}$, \\
      $~~~~~$the iteration counter $k=1$;\\
      ~2:~~Take $\lambda=\|(\Phi^T(\mathbf{z}-\Phi\mathbf{v}))_{S^{c}}\|_\infty$ and $S=S\cup\arg\max_{j}\{|\varPhi_j^T(\mathbf{z}-\Phi \mathbf{v})|, j\in S^{c}\}$; \\
      ~3:~~Update the solution with $\mathbf{v}_S=(\Phi_S^T\Phi_S)^{-1}(\Phi_S^T\mathbf{z}-\lambda\mathbf{s}_S)$ on the support, where $\mathbf{s}_S=\text{sign}(\mathbf{v}_S)$; \\
      ~4:~~If $i\in S$, s.t. $v_i=0$, let $S=S\setminus\{i\in S;v_i=0\}$; \\
      ~5:~~$k\rightarrow k+1$;\\
      ~6:~~Return to Step 2 if $\lambda >\lambda_0$, otherwise \textbf{terminate};\\
      ~7:~~$\mathbf{v}_S=(\Phi_S^T\Phi_S)^{-1}(\Phi_S^T\mathbf{z})$, $w_{\mathbf{N}}=\sum_{i=1}^{\mathbf{N}}\mathbf{v}(i)\phi_i(\bm{\xi})$.
      \label{algorithm-ILARS}
\end{algorithm}

At each iteration in Algorithm \ref{algorithm-ILARS}, the choice of regularization parameter $\lambda$ is based on Theorem \ref{theorem-lambda}.
We note that the number of steps required by ILARS is no more than the dimensions of $\mathbf{v}$.

 In order to select the optimal regularization parameter $\lambda_0$, the classical $k-$fold cross validation method is usually considered. However, the $k-$fold cross validation method may be time-consuming. This may be computationally expensive when applying an iterative strategy, which is the case in the paper. To overcome this difficulty, we apply a fast leave-one-out method \cite{chevreuil2015least} to determine optimal ILARS solution that only requires a single call to the Algorithm \ref{algorithm-ILARS} with a proper regularization parameter.
The fast leave-one-out method is described in Algorithm \ref{algorithm-FILARS}. For simplicity of presentation, we use the abbreviation FILARS to denote ILARS using fast leave-one-out method.

\begin{algorithm}
\caption{Fast leave-one-out method to determine optimal ILARS solution}
     \textbf{Input}: A matrix $\Phi$, the vector $\mathbf{z}$, and the given relatively small regularization parameter $\lambda_0$\\
     \textbf{Output}: The optimal ILARS solution $\mathbf{v}$, the solution support $S$, and the approximation \\ $w_{\mathbf{N}}=\sum_{i=1}^{\mathbf{N}}\mathbf{v}(i)\phi_i(\bm{\xi})$ \\
      ~1:~~Run Algorithm \ref{algorithm-ILARS} with $\lambda_0$ one time to obtain $k$ solutions $\mathbf{v}_1,\cdots,\mathbf{v}_k$, with corresponding \\
      $~~~~~$sets of nonzero coefficients $S_1,\cdots,S_k$;\\
      ~2:~$~$\textbf{for} $j=1,\cdots,k$ \textbf{do}\\
      ~3: ~~~~Correct the nonzero coefficients $\mathbf{v}_j$ with $\mathbf{v}_{S_j}^j=(\Phi_{S_j}^T\Phi_{S_j})^{-1}(\Phi_{S_j}^T\mathbf{z})$;\\
      ~4:~~~~~Compute $h_q=(\Phi_{S_j}(\Phi_{S_j}^T\Phi_{S_j})^{-1}\Phi_{S_j}^T)_{qq}$; \\
      ~5:~~~~~Compute relative leave-one-out error $\epsilon_j=\frac{1}{M}\sum_{q=1}^M\bigg(\frac{(\mathbf{z})_q-\Phi_{S_j}\mathbf{v}_{S_j}^j}{(1-h_q)\sigma(\mathbf{z})}\bigg)^2$, where $\sigma(\mathbf{z})$ is the\\
      ~6:~~~~~empirical standard deviation of $\mathbf{z}$; \\
      ~7:~$~$\textbf{end for}\\
      ~8:~~Select optimal solution $\mathbf{v}$ such that $\mathbf{v}=\mathbf{v}_{S_{j^*}}^{j^*}$ with $j^*=\arg\min_j \epsilon_j$.
      \label{algorithm-FILARS}
\end{algorithm}

When dealing with the problems in a high stochastic dimension space, the dimension of approximation space $\mathcal{S}_\mathbf{N}$ grows exponentially with the dimension  of the stochastic variable $\bm{\xi}$, which makes it difficulty to get a good approximation of the model output $w(\bm{\xi})$ by FILARS directly.
To overcome the difficulty, in the next section  we will introduce a hierarchical sparse low rank tensor approximation based on the FILARS method by decomposing  a high dimensional stochastic problem   into some low dimensional
stochastic problems.

\section{Hierarchical sparse low rank tensor approximation}
\label{hierarchical-low-rank-app}

In this section, we introduce the HSLRTA method, which can be used to approximating $\zeta_i(\bm{\xi})$ in (\ref{exi}) with high  dimensional random variable $\bm{\xi}$.
We attempt to seek a sparse rank-$m$ approximation of the model output $w(\bm{\xi})$ under the form
\begin{eqnarray}
\label{hi-low-rank}
 w_\mathbf{N}(\bm{\xi})=\sum_{i=1}^{m}\alpha_iv_i(\bm{\xi})=\sum_{i=1}^{m}\alpha_i\prod_{k=1}^rw^{(i,k)}(\bm{\xi}_k),
\end{eqnarray}
in the finite dimensional tensor space $\mathcal{S}_\mathbf{n}=\mathcal{S}_{n_1}^1\otimes\cdots\otimes\mathcal{S}_{n_r}^r$, where
\[
w^{(i,k)}(\bm{\xi}_k)=\sum_{j=1}^{n_k}w_j^{(i,k)}\phi_j^k(\bm{\xi}_k)=\bm{\phi}^k\mathbf{w}^{(i,k)}\in \mathcal{S}_{n_k}^k ~~\text{for} ~~i=1,\cdots,m.
\]
Here $\mathbf{w}^{(i,k)}$ denotes the vector of coefficients of $w^{(i,k)}$ with
only a few nonzero coefficients, and $\bm{\phi}^k=(\phi_1^k,\cdots,\phi_{n_k}^k)$ denotes the vector of basis functions.
In Subsection \ref{S-low-rank-tensor}, we construct such an approximation by successively computing sparse rank-one approximation, i.e., $v_i(\bm{\xi})=\prod_{k=1}^rw^{(i,k)}(\bm{\xi}_k)$.

\subsection{Sparse low rank tensor approximation}
\label{S-low-rank-tensor}
We denote $\mathcal{R}_1$ the set of (elementary) rank-one tensors in $\mathcal{S}_\mathbf{n}=\mathcal{S}_{n_1}^1\otimes\cdots\otimes\mathcal{S}_{n_r}^r$, i.e.,
\[
\mathcal{R}_1=\bigg\{v=\big(v^{(1)}\otimes\cdots\otimes v^{(r)}\big)(\bm{\xi})=\prod_{k=1}^rv^{(k)}(\bm{\xi}_k)=\prod_{k=1}^r\bm{\phi}^k\mathbf{v}^k: v^{(k)}(\bm{\xi}_k) \in\mathcal{S}_{n_k}^k \bigg\},
\]
where $\mathbf{v}^k \in \mathbb{R}^{n_k}$ denotes the vector of coefficients of $v^{(k)}$.
Let $\mathcal{R}_m$ be the set of (elementary) rank-$m$ tensors
\[
\mathcal{R}_m=\bigg\{\sum_{i=1}^{m}\alpha_iv_i(\bm{\xi}): v_i(\bm{\xi}) \in \mathcal{R}_1 \bigg\}.
\]
For $k=1,\cdots,r$, $\bm{\phi}^k=(\phi_1^k,\cdots,\phi_{n_k}^k)$ denotes the vector of basis functions.
We compute a sparse rank-one approximation $v=\prod_{k=1}^rv^{(k)}(\bm{\xi}_k)\in \mathcal{R}_1$ of $w$ by solving the following $l_1$-optimization problem:
\begin{eqnarray}
\label{lowrank-optimization-1}
\begin{split}
\min_{v\in \mathcal{R}_1}\|w-v\|_{L^2}^2~~\text{subject to}~~\|\mathbf{v}^1\|_1\leq \delta_1,\cdots,\|\mathbf{v}^r\|_1\leq \delta_r,
\end{split}
\end{eqnarray}
where $\mathbf{v}^1\in \mathbb{R}^{n_1},\cdots,\mathbf{v}^r\in \mathbb{R}^{n_r}$ and $v=(\bm{\phi}^1(\bm{\xi}_1)\mathbf{v}^1)\otimes\cdots\otimes(\bm{\phi}^{(r)}(\bm{\xi}_r)\mathbf{v}^r)$.
We can consider the following optimization problem equivalent to  (\ref{lowrank-optimization-1}),
\begin{eqnarray}
\label{lowrank-optimization-1lambda}
\begin{split}
\min_{v\in \mathcal{R}_1}\|w-v\|_{L^2}^2+\sum_{k=1}^r\lambda_k\|\mathbf{v}^k\|_1,
\end{split}
\end{eqnarray}
where the regularization parameters $\lambda_k > 0$ (Lagrange multipliers) are related to $\delta_k$.
We can solve the optimization problem (\ref{lowrank-optimization-1lambda}) by an alternating minimization algorithm.

For $k=1,\cdots,r-1$, we are devoted to constructing the vector $\mathbf{v}^k$ of coefficients of $v^{(k)}$ by solving the following optimization problem
 \begin{eqnarray}
\label{k-optimization-1lambda}
\begin{split}
\min_{\mathbf{v}^k\in \mathbb{R}^{n_k}}\|\mathbf{z}^k-\bm{\Phi}^{(k)}\mathbf{v}^k\|_2^2+\lambda_k\|\mathbf{v}^k\|_1,
\end{split}
\end{eqnarray}
where $\mathbf{z}^k\in \mathbb{R}^{n_k}$ denotes the vector of random evaluations of $w(\bm{\xi})$ corresponding to $M^k$ samples of $\bm{\xi}$ such that, for each $j\in\{i\in\{1,\cdots,r\}: i\neq k$\}, $\bm{\hat{\xi}}_j$ is a fixed sample of  dimension $d_j$, and $\{\bm{\xi}_k^{(i)}\}_{i=1}^{M^k}\in\Omega_k$
 are $M^k$ different samples of  dimension $d_k$. The $\ell_1$ optimization problem  (\ref{k-optimization-1lambda})
  can be solved by FILARS in Algorithm \ref{algorithm-FILARS}. Suppose  that $\hat{\mathbf{v}}^k$ is the sparse solution of the optimization problem (\ref{k-optimization-1lambda}), then we let $w^{(k)}(\bm{\xi}_k)=\bm{\phi}^k(\bm{\xi}_k)\hat{\mathbf{v}}^k$.
 For $k=r$, we solve the following optimization problem
 \begin{eqnarray}
\label{k-optimization-rlambda}
\begin{split}
\min_{\mathbf{v}^r\in \mathbb{R}^{n_r}}\|\mathbf{z}^r-\bm{\hat{\Phi}}^{(r)}\mathbf{v}^r\|_2^2+\lambda_r\|\mathbf{v}^r\|_1,
\end{split}
\end{eqnarray}
which is used to  construct the vector of coefficients of $v^{(r)}$, i.e., $\mathbf{v}^r\in \mathbb{R}^{n_r}$,
 where $\mathbf{z}^r\in \mathbb{R}^{M^r}$ denotes the vector of random evaluations of $w(\bm{\xi})$ corresponding to $M^r$ samples of $\bm{\xi}$ such
 that, for each $j\in\{1,\cdots,r-1\}$, $\bm{\hat{\xi}}_j$ is a fixed sample of  dimension $d_j$, and $\{\bm{\xi}_r^{(i)}\}_{i=1}^{M^r}$
 are $M^r$ different samples of  dimension $d_r$, and $\bm{\hat{\Phi}}^{(r)}$ is the matrix with components
\begin{eqnarray}
\label{matrix-r}
\begin{split}
 (\bm{\hat{\Phi}}^{(r)})_{i,j}=(\bm{\Phi}^{(r)})_{i,j}\prod_{k=1}^{r-1}w^{(k)}(\bm{\hat{\xi}}_k).
\end{split}
\end{eqnarray}
Suppose $\hat{\mathbf{v}}^r$ is the sparse solution of the optimization problem (\ref{k-optimization-rlambda}), which is  solved by FILARS in Algorithm \ref{algorithm-FILARS}.
We take $w^{(r)}(\bm{\xi}_r)=\bm{\phi}^r(\bm{\xi}_r)\hat{\mathbf{v}}^r$.
Then the sparse rank-one approximations  $w_1(\bm{\xi})$ of $w(\bm{\xi})$ can be expressed by
\[
w_1(\bm{\xi})=\prod_{k=1}^{r}w^{(k)}(\bm{\xi}_k)=\prod_{k=1}^{r}\bm{\phi}^k(\bm{\xi}_k)\hat{\mathbf{v}}^k.
\]
We summarize the main steps to construct  a sparse rank-one approximation in Algorithm \ref{algorithm-rank1}.

\begin{algorithm}
\caption{The construction of a sparse rank-one approximation of a  model output $w(\bm{\xi})$}
     \textbf{Input}: Vectors of evaluations $\{\mathbf{z}^k\}_{k=1}^r$, basis matrices $\{\bm{\Phi}^{(k)}\}_{k=1}^{r-1}$\\
     \textbf{Output}: The sparse rank-one approximations $w_1(\bm{\xi})=\prod_{k=1}^{r}\bm{\phi}^k(\bm{\xi}_k)\hat{\mathbf{v}}^k$\\
      ~1:~$~$\textbf{for} $k=1,\cdots,r-1$ \textbf{do}\\
      ~2: ~~~~Select the optimal regularization parameter $\lambda_k$ using modified cross validation;\\
      ~3:~~~~~Solve the optimization problem (\ref{k-optimization-1lambda}) by  Algorithm \ref{algorithm-FILARS} with $\mathbf{z}^k$ and\\ $~~~~~~~~$$\bm{\Phi}^{(k)}$ to obtain the optimal sparse solution $\hat{\mathbf{v}}^k$ and $w^{(k)}(\bm{\xi}_k)=\bm{\phi}^k(\bm{\xi}_k)\hat{\mathbf{v}}^k$;\\
      ~4:~$~$\textbf{end for}\\
      ~5:~~Construct the matrix $\bm{\hat{\Phi}}^{(r)}$ with equation (\ref{matrix-r}); \\
      ~6:~~Solve the optimization problem (\ref{k-optimization-rlambda}) by   Algorithm \ref{algorithm-FILARS} with $\mathbf{z}^r$ and $\bm{\Phi}^{(r)}$ \\
      $~~~~~$to obtain the optimal sparse solution $\hat{\mathbf{v}}^r$ and $w^{(r)}(\bm{\xi}_k)=\bm{\phi}^r(\bm{\xi}_r)\hat{\mathbf{v}}^k$; \\
      ~7:~~Get the sparse rank-one approximations of $w(\bm{\xi})$: \\
      $~~~~~$$w_1(\bm{\xi})=\prod_{k=1}^{r}w^{(k)}(\bm{\xi}_k)=\prod_{k=1}^{r}\bm{\phi}^k(\bm{\xi}_k)\hat{\mathbf{v}}^k$.\\
      \label{algorithm-rank1}
\end{algorithm}
\begin{rem}
We can get different rank-one approximations by changing the  regularization of the optimization problem
(\ref{k-optimization-1lambda}) and (\ref{k-optimization-rlambda}), the ordinary  least-squares and regularized least-squares
 will be used in the numerical examples in
 Section \ref{numerical examples}.
\end{rem}

 Now we want to construct a sparse rank-$m$ approximation $w_m\in \mathcal{R}_m$ of $w$ under the form (\ref{hi-low-rank}). Suppose $w_0=0$, and the approximation
$w_{i-1}$ of $w$ is given. Such an approximation can be constructed by successively computing the sparse rank-one approximation problems as follows:
for $i=1,\cdots,m$,
\begin{eqnarray}
\label{lowrank-optimization-mlambda}
\begin{split}
\min_{v\in \mathcal{R}_1}\|w-&w_{i-1}-v\|_{L^2}^2+\sum_{k=1}^r\lambda_k\|\mathbf{v}^k\|_1,\\
\text{where}~~\mathbf{v}^1\in \mathbb{R}^{n_1},\cdots,\mathbf{v}^r\in &\mathbb{R}^{n_r}~~\text{and}~~ v=(\bm{\phi}^1(\bm{\xi}_1)\mathbf{v}^1)\otimes\cdots\otimes(\bm{\phi}^{r}(\bm{\xi}_r)\mathbf{v}^r).
\end{split}
\end{eqnarray}
Problem (\ref{lowrank-optimization-mlambda}) can be solved by Algorithm \ref{algorithm-rank1}, where $\{\mathbf{z}^k\}_{k=1}^r$ are the vectors of
evaluations of $(w-w_{i-1})(\bm{\xi})$.
We provide the details of the construction of a sparse rank-$m$ approximation in Algorithm \ref{algorithm-rankm}.
\begin{algorithm}
\caption{The construction of a sparse rank-$m$ approximation of a  model output $w(\bm{\xi})$}
     \textbf{Input}: Maximal rank $m$, vectors of evaluations $\{\mathbf{z}^{(k,i)}\}_{k=1, i=1}^{r,m}$, basis matrices $\{\bm{\Phi}^{(k)}\}_{k=1}^{r}$,\\
     and and the error tolerance $\varepsilon^*$\\
     \textbf{Output}: The sparse rank-$m$ approximations: \\
     $w_m(\bm{\xi})=\sum_{i=1}^mv_i(\bm{\xi})=\sum_{i=1}^m\prod_{k=1}^{r}\bm{\phi}^{(k,i)}(\bm{\xi}_k)\hat{\mathbf{v}}^{(k,i)}$\\
      ~1:~$~$\textbf{Initialization}: Set $w_0=0$; \\
      ~2:~$~$\textbf{for} $i=1,\cdots,m$ \textbf{do}\\
      ~3: ~~~~Evaluate the vectors $\{\mathbf{z}^k_{i-1}\}_{k=1}^{r}$ of evaluations of $w_{i-1}$;\\
      ~4:~~~~~Compute the sparse rank-one approximation $v_i=\prod_{k=1}^{r}\bm{\phi}^{(k,i)}(\bm{\xi}_k)\hat{\mathbf{v}}^{(k,i)}$ by Algorithm \ref{algorithm-rank1} \\
      $~~~~~~~~$with $\{\mathbf{z}^{(k,i)}-\mathbf{z}^k_{i-1}\}_{k=1}^r$ and basis matrices $\{\bm{\Phi}^{(k)}\}_{k=1}^{r-1}$;\\
      ~5:~~~~Get the sparse rank-$i$ approximations $w_i(\bm{\xi})=w_{i-1}(\bm{\xi})+v_{i}(\bm{\xi})$ of $w$;\\
      ~7:~~~~Set $\varepsilon=E(v_{i}(\bm{\xi}))$;\\
      ~8:~~~~\textbf{if}  $\varepsilon < \varepsilon^{*}$;\\
      ~9~~~~~~~$m=i$;\\
      ~10:~~~\textbf{end if}\\
      ~6:~$~$\textbf{end for}\\
      ~7:~~Get the sparse rank-$m$ approximations of $w(\bm{\xi})$:
        $w_m(\bm{\xi})=\sum_{i=1}^m\prod_{k=1}^{r}\bm{\phi}^{(k,i)}(\bm{\xi}_k)\hat{\mathbf{v}}^{(k,i)}$.\\
      \label{algorithm-rankm}
\end{algorithm}

\begin{rem}
Once the sequence of sparse rank-one approximations $\{v_i\}_{i=1}^M$ have been computed, $w_m=\sum_{i=1}^mv_i(\bm{\xi})$ usually gives
a good approximation. If better approximation is required,  we can make a correction for the sparse rank-$m$ approximation, i.e.,
take $w_m=\sum_{i=1}^m\beta_iv_i(\bm{\xi})$ instead of $w_m=\sum_{i=1}^mv_i(\bm{\xi})$. Here $w_m=\sum_{i=1}^m\beta_iv_i(\bm{\xi})$ can be computed by solving the following optimization problem:
\[
\min_{\bm{\beta}\in \mathbb{R}^m}\|w-\sum_{i=1}^m\beta_iv_i(\bm{\xi})\|_1+\lambda\|\bm{\beta}\|_1.
\]
\end{rem}

\begin{rem}
As for the choice of samples of $\bm{\xi}$ and the construction of the inputs in Algorithm \ref{algorithm-rankm},
we firstly  take $\{\bm{\xi}_k^{(i)}\}_{i=1}^{M^k}\in\Omega_k$
to construct the basis matrices
$\{\bm{\Phi}^{(k)}\}_{k=1}^r$,   and then  take $m$ samples $\bm{\hat{\xi}}^{(i)}=(\bm{\hat{\xi}}_1^{(i)},\cdots,\bm{\hat{\xi}}_r^{(i)}),~i=1,\cdots, m$. For each $i=1,\cdots,m$ and each $k=1,\cdots,r$, we
construct $\mathbf{z}^{(k,i)}$ based on a set of samples such that
$\bm{\xi}^{(j)}=(\bm{\hat{\xi}}_1^{(i)},\cdots,\bm{\hat{\xi}}_1^{(j-1)}, \bm{\xi}_k^{(j)}, \bm{\hat{\xi}}_1^{(j+1)}\cdots,  \bm{\hat{\xi}}_r^{(i)})$, $j=1, \cdots, M_k$.
\end{rem}

\begin{rem}
If the iteration procedure is terminated by the maximal rank $m$ instead of the error tolerance $\varepsilon^*$, Algorithm \ref{algorithm-rankm}
may  not give  a better approximation by a tensor approximation with a rank higher than $m$.
 For this situation, we can select an optimal  rank using
cross validation method.
\end{rem}

We may not get a good approximation of the model output $w(\bm{\xi})$ by  Algorithm \ref{algorithm-rankm} in general when the best rank $m$
is large and $r$ (the number of subsets of the random variables $\bm{\xi}$) is larger than $3$. In order to overcome the difficulty, we introduce
a hierarchical sparse low rank tensor approximation method, where the ``hierarchical" means hierarchical Tucker tensor sets.

\subsection{Hierarchical sparse low rank tensor approximation}
\label{HSLRTA}

In the approach of the hierarchical tensor formulation, we repeatedly  use  the concept of tensor subspaces in higher levels, and  divide the subspaces
in a hierarchical manner  so that the dimension is reduced. The recursive use of the subspaces  leads to a  tree structure describing the hierarchy
of subspaces. Here we consider the linear tree $T_D^{TT}$ depicted in Figure \ref{fig-hierical}.
\begin{figure}[htbp]
\centering
  \includegraphics[width=5 in, height=2.3in]{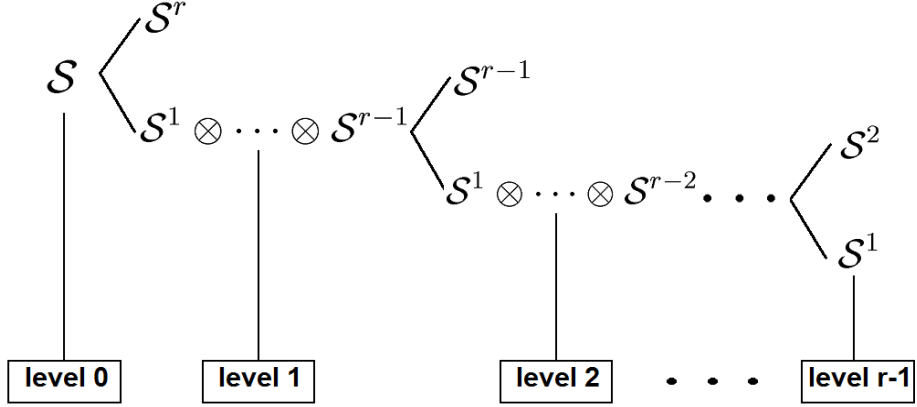}
  \caption{The linear tree $T_D^{TT}$.}
  \label{fig-hierical}
\end{figure}
In this case,  the largest level number (depth of the tree) is $r-1$. Without loss of the generality, we present the HSLRTA for a generic multivariate function $w(\bm{\xi}_1,\cdots,\bm{\xi}_r)\in \mathcal{S}$.
Based on the structure of  hierarchical tensor, the idea of HSLRTA can be described  as follows.

Given $\mathcal{S}=\mathcal{X}_1\otimes\mathcal{S}^r$, where  $\mathcal{X}_1=\mathcal{S}^1\otimes\cdots\otimes\mathcal{S}^{r-1}$,
then $w(\bm{\xi}_1,\cdots,\bm{\xi}_r)$ can be approximated by
\begin{eqnarray}
\label{Hlowrank-1}
\begin{split}
w(\bm{\xi}_1,\cdots,\bm{\xi}_r)=\sum_{i_1=1}^{m_1}w_{1}^{i_1}(\bm{\xi}_1,\cdots,\bm{\xi}_{r-1})v_1^{i_1}(\bm{\xi}_r)~~~~~~\longrightarrow\text{level} ~1.
\end{split}
\end{eqnarray}
 Because $w_{1}^{i_1}(\bm{\xi}_1,\cdots,\bm{\xi}_{r-1})\in \mathcal{X}_1$ ($i_1=1,\cdots,m_1$) and $\mathcal{X}_1=\mathcal{X}_2\otimes\mathcal{S}^{r-1}$ ($\mathcal{X}_2=\mathcal{S}^1\otimes\cdots\otimes\mathcal{S}^{r-2}$), then $w_1^{i_1}(\bm{\xi}_1,\cdots,\bm{\xi}_r)$ can be approximated by
\begin{eqnarray}
\label{Hlowrank-2-1}
\begin{split}
w_1^{i_1}(\bm{\xi}_1,\cdots,\bm{\xi}_{r-1})=\sum_{i_2=1}^{m_2}w_{2}^{i_2}(\bm{\xi}_1,\cdots,\bm{\xi}_{r-2})v_2^{i_2}(\bm{\xi}_{r-1})  ~~~~~~\longrightarrow\text{level} ~2.
\end{split}
\end{eqnarray}
Consequently,  (\ref{Hlowrank-1}) can be rewritten  as
\begin{eqnarray}
\label{Hlowrank-2}
\begin{split}
w(\bm{\xi}_1,\cdots,\bm{\xi}_r)=\sum_{i_1=1}^{m_1}\big(\sum_{i_2=1}^{m_2}w_{2}^{i_2}(\bm{\xi}_1,\cdots,\bm{\xi}_{r-2})v_2^{i_2}(\bm{\xi}_{r-1})\big)^{\{i_1\}}v_1^{i_1}(\bm{\xi}_r).
\end{split}
\end{eqnarray}
We can similarly  get
\begin{eqnarray}
\label{Hlowrank-r-1-1}
\begin{split}
w_{r-2}^{i_{r-2}}(\bm{\xi}_1,\bm{\xi}_2)=\sum_{i_{r-1}=1}^{m_{r-1}}w_{r-1}^{i_{r-1}}(\bm{\xi}_1)v_{r-1}^{i_{r-1}}(\bm{\xi}_2)  ~~~~~~~~\longrightarrow\text{level} ~r-1.
\end{split}
\end{eqnarray}
Thus,  (\ref{Hlowrank-1}) can be expressed by
\begin{eqnarray}
\label{Hlowrank-r-1}
\begin{split}
&w(\bm{\xi}_1,\cdots,\bm{\xi}_r)=\\
&\sum_{i_1=1}^{m_1}\bigg(\sum_{i_2=1}^{m_2}\cdots\sum_{i_{r-2}=1}^{m_{r-2}}\big(\sum_{i_{r-1}=1}^{m_{r-1}}w_{r-1}^{i_{r-1}}(\bm{\xi}_1)v_{r-1}^{i_{r-1}}
(\bm{\xi}_2)\big)^{\{i_{r-2}\}}v_{r-2}^{i_{r-2}}
(\bm{\xi}_3)\cdots v_2^{i_2}(\bm{\xi}_{r-1})\bigg)^{\{i_1\}}v_1^{i_1}(\bm{\xi}_r).
\end{split}
\end{eqnarray}
In this subsection, we  seek a sparse approximation of $w(\bm{\xi})$ under the form (\ref{Hlowrank-r-1}).

Here we describe the main steps  of HSLRTA as follows:\\
$\bullet$ \textit{Step 1:} At the maximal level $r-1$, for each $k=1,\cdots,m_{r-2}$, we use Algorithm \ref{algorithm-rankm} to compute the sparse rank-$m_{r-1}$ approximation of $w_{r-2}^{k}(\bm{\xi}_1,\bm{\xi}_2)$, i.e.
$$w_{r-2}^{k}(\bm{\xi}_1,\bm{\xi}_2)=\big(\sum_{i_{r-1}=1}^{m_{r-1}}w_{r-1}^{i_{r-1}}(\bm{\xi}_1)v_{r-1}^{i_{r-1}}(\bm{\xi}_2)\big)^{\{k\}}.$$
$\bullet$  \textit{Step 2:} At the level $r-2$, for each $i_{r-3}=1,\cdots,m_{r-3}$, we just need to solve the following optimization problem \big(suppose  that the sparse rank-$(k-1)$ approximation of $w_{r-3}^{i_{r-3}}(\bm{\xi}_1,\bm{\xi}_2,\bm{\xi}_3)$, i.e., $\hat{w}_{k-1}=\sum_{i=1}^{k-1} w_{r-2}^{i}(\bm{\xi}_1, \bm{\xi}_2)v_{r-2}^{i}(\bm{\xi}_3)$, has been obtained, and start with $\hat{w}_0=0$, where $1\leq k \leq m_{r-2}$.\big)
for the fixed sample $\hat{\xi}_1 \in \Omega_1$, $\hat{\xi}_2 \in \Omega_2$,
 \begin{eqnarray}
\label{H-lowrank-optimization-mlambda}
\begin{split}
\min_{v\in \mathcal{R}_1}\|w_{r-3}^{i_{r-3}}&(\hat{\xi}_1,\hat{\xi}_2,\bm{\xi}_3)-\hat{w}_{k-1}-w_{r-2}^{k}(\hat{\xi}_1,\hat{\xi}_2)v\|_{L^2}^2+\lambda\|\mathbf{v}\|_1,\\
&\text{where}~~\mathbf{v}\in \mathbb{R}^{n_3},~\text{and}~~ v=\bm{\phi}^3(\bm{\xi}_3)\mathbf{v},\\
\end{split}
\end{eqnarray}
to get $v_{r-2}^{k}(\bm{\xi}_3)\bm{\phi}^3(\bm{\xi}_3)\tilde{\mathbf{v}}$, where $\tilde{\mathbf{v}}$ is the solution of the optimization problem (\ref{H-lowrank-optimization-mlambda}),  and get the the sparse rank-$k$ approximation of $w_{r-3}^{i_{r-3}}(\bm{\xi}_1,\bm{\xi}_2,\bm{\xi}_3)$, i.e.,  $\hat{w}_k=\sum_{i=1}^{k} w_{r-2}^{i}(\bm{\xi}_1, \bm{\xi}_2)v_{r-2}^{i}(\bm{\xi}_3)$. Thus we get the sparse rank-$m_{r-2}$ approximation of $w_{r-3}^{i_{r-3}}(\bm{\xi}_1,\bm{\xi}_2,\bm{\xi}_3)$, i.e.,
\[
w_{r-3}^{i_{r-3}}(\bm{\xi}_1,\bm{\xi}_2,\bm{\xi}_3)=\bigg(\sum_{i=1}^{m_{r-2}} \big(\sum_{i_{r-1}=1}^{m_{r-1}}w_{r-1}^{i_{r-1}}(\bm{\xi}_1)v_{r-1}^{i_{r-1}}(\bm{\xi}_2)\big)^{\{i\}} v_{r-2}^{i}(\bm{\xi}_3)\bigg)^{\{i_{r-3}\}}.
\]
$\bullet$ \textit{Step 3:} Repeat the \textit{Step 2} until the procedure is back up to the level 1, and we get the sparse approximation of $w(\bm{\xi})$ under the form (\ref{Hlowrank-r-1}).

\section{Numerical results}
\label{numerical examples}

In this section, we present a few examples to illustrate the performance of the proposed methods and make some comparisons for different strategies. In each example we seek a function representation approximation to a model output given a set of uncertain parameters with a known range or distribution. In Section \ref{Numerical-result1}, we use an example to illustrate the performance of improved least angle regression algorithm and the hierarchical sparse low rank tensor approximation (HSLRT), and demonstrate  the advantages of FILARS and HSLRT over  the ordinary least-squares (OLS) and orthogonal matching pursuit (OMP), which are well-known and widely used.  In Section \ref{Numerical-result2}, we consider a multivariate function dependent on the random variables $\bm{\xi}$ and the spatial variables $x$  to present the performance of novel variable-separation (NVS),  where we use FILARS to get a good approximation of $\zeta_i(\bm{\xi})$ in (\ref{separat-presentation}) for each $i=1,\cdots,N$.
In Section \ref{Numerical-result3}, an elliptic PDE with a high dimensional parameter is considered. NVS is used to get an approximation in the form (\ref{separat-presentation}) for the solution to  the elliptic PDE, then HSLRT is used to uncouple the dependence between different terms $\{\zeta_i(\bm{\xi})\}_{i=1}^{N}$ for uncertainty quantification.

\subsection{Rastrigin function}
\label{Numerical-result1}
In this subsection, we consider a function example to illustrate the performance of the proposed numerical algorithms in this work.
Since the wavelet basis is  able to  simultaneously capture the global and local features, it is a good choice for the approximation of oscillating functions.
In this example, we construct an approximation of the  Rastrigin function with $6$ variables by solving three different optimization problems: the ordinary least-squares problem, $l_0$-norm optimization  and $l_1$-norm optimization. The orthogonal polynomials basis functions are used in the approximation. The Rastrigin function is given by
\[
w(\bm{\xi})=60+\sum_{i=1}^6\big(\xi_i^2-10\cos{(2\pi\xi_i)}\big),
\]
where $\xi_1,\cdots,\xi_6 $ are independent random variables and  uniformly distribute in $[-1, 1]$. We note that the number of random sample evaluations required by OLS scales quadratically with the number $\mathbf{N}+1$ of the polynomial basis (Legendre polynomials for this example), i.e., $M \thicksim (\mathbf{N} +1)^2$, where $M$ is the number of samples. If the solution of OLS problem is sparse,  a good approximation can be obtained by solving the $l_0$-norm optimization problem (\ref{sparse-optimization-0}) with only $M \ll \mathbf{N}$ random samples. Here we  use OMP to solve the $l_0$-norm optimization problem. As we know, $l_0$-norm optimization problem is an NP-hard problem, it is will cost too much to solve the $l_0$-norm optimization problem when the number of orthogonal polynomials basis functions is too large. Since the $l_0$-norm optimization problem (\ref{sparse-optimization-0}) can be reasonably approximated by the $l_1$-norm optimization problem under certain conditions, we solve the $l_1$-norm convex optimization problem instead of $l_0$-norm optimization problem to approximate the solution. In this subsection, the ILARS introduced in Section \ref{mlars} is used to solve the $l_1$-norm optimization problem, then the fast leave-one-out method is used to determine the optimal ILARS solution (FILARS).
In order to make comparison for all methods, we also apply the hierarchical sparse low rank tensor approximation  (HSLRTA) method to this example. For HSLRTA, we split $\bm{\xi}=(\xi_1,\cdots, \xi_{6})$ into $3$ mutually independent sets $\{\bm{\xi}_i=(\xi_{2i-1}, \xi_{2i})\}_{i=1}^3$ of random variables, then the largest level number is $2$ according to Section \ref{HSLRTA}. Here we solve the $l_1$-norm optimization problem (\ref{sparse-optimization-1}) by solving several sub-optimization problems,   whose dimension of the random variables is only $2$.

\begin{figure}[htbp]
\centering
  \includegraphics[width=3.2in, height=2.5in]{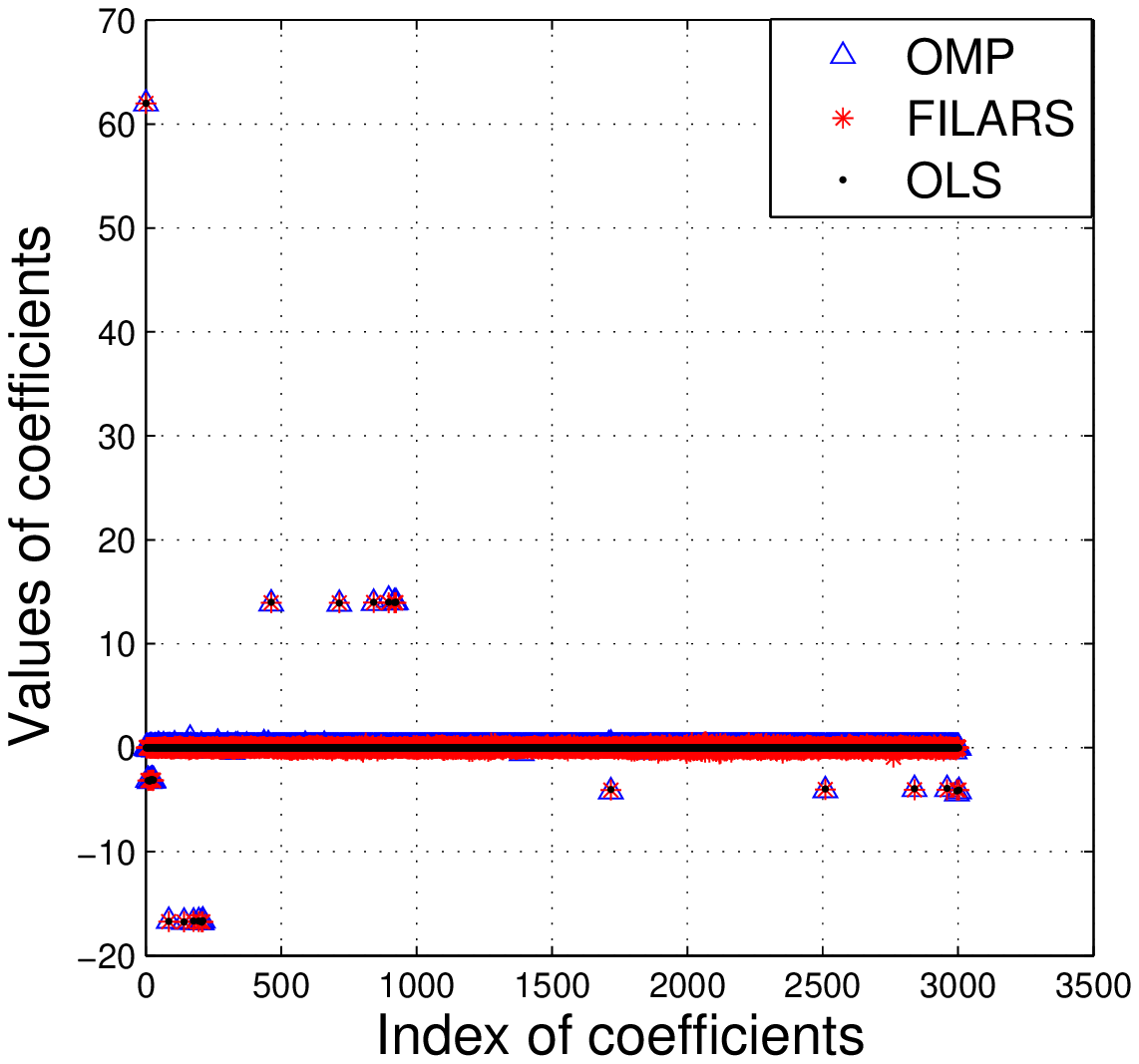}
  \includegraphics[width=3.2in, height=2.5in]{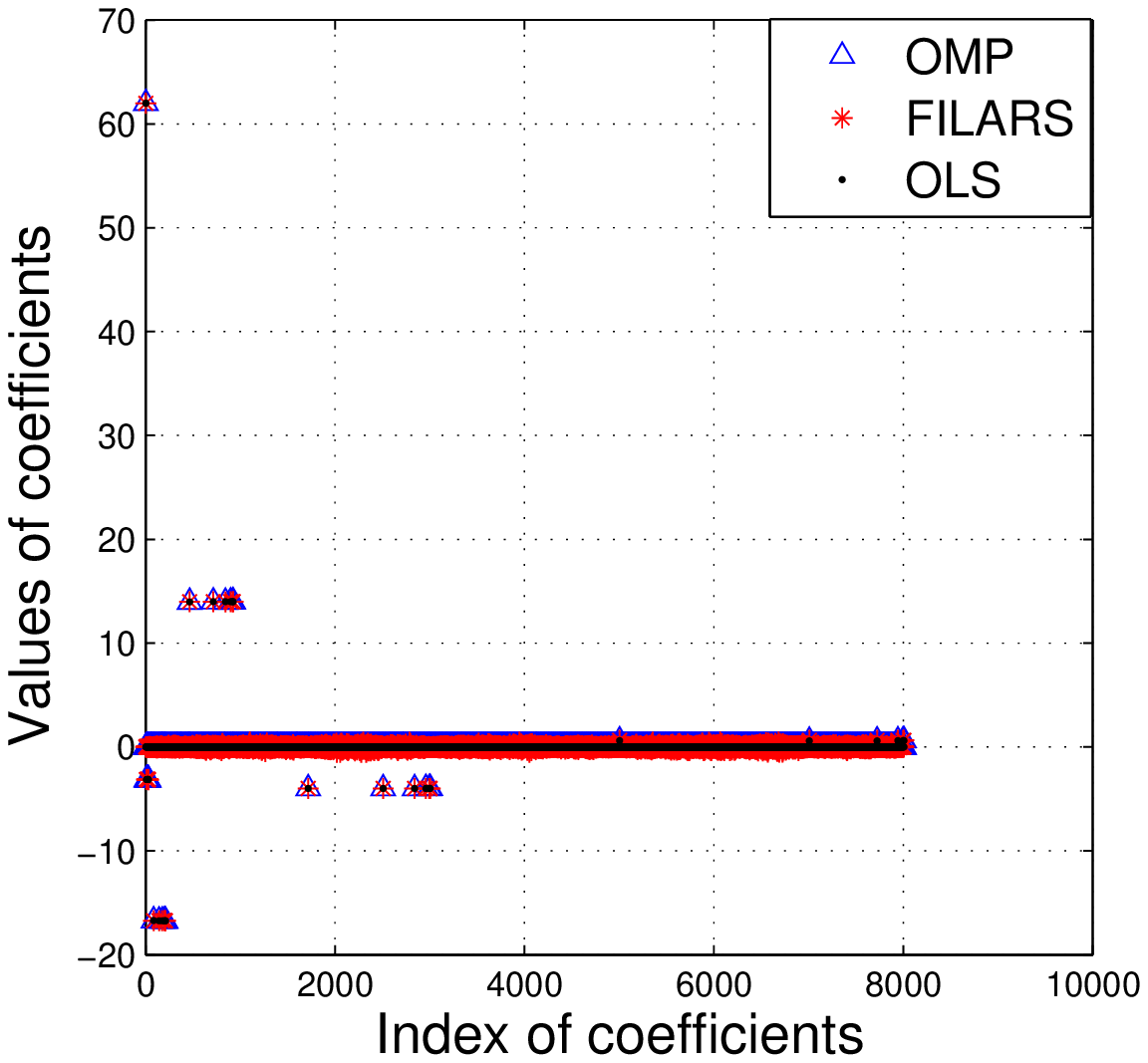}
  \caption{Comparison of coefficients corresponding to the Legendre polynomial basis function using OLS, OMP, and FILARS  for the $6$ dimensional Rastrigin function, the maximal polynomial degrees are $p=8$ (left) and $p=10$ (right), respectively.}
  \label{fig1-exam1}
\end{figure}

\begin{table}[hbtp]
\centering
\caption{ Comparison of the relative mean errors $\varepsilon$, the number of samples $M$, and the offline CPU time for different approaches (OLS, OMP, FILARS and HSLRTA) and different polynomial degrees $p$.}
\begin{tabular}{c|c|c|c|c}
\Xhline{1pt}
  Strategies & $p$ & $\varepsilon$ & offline CPU time & M\\
  \hline
  \multirow{3}{*}{OLS}
                & $8$ & $7.60\times10^{-3}$ & $1.1830~s$ & $7000$\\
                & $10$ & $2.90\times10^{-3}$ & $5.9602~s$& $15000$ \\
                & $12$ & $2.33\times10^{-4}$ & $44.0064~s$ & $22000$\\
 \hline
 \multirow{3}{*}{OMP}   & $8$ & $7.30\times10^{-3}$ & $0.8005~s$& $320$ \\
                        & $10$ & $6.00\times10^{-4}$ & $2.1443~s$& $500$ \\
                        & $12$ & $3.09\times10^{-5}$ & $6.3213 ~s$ & $620$\\
 \hline
 \multirow{3}{*}{ILARS}
                & $8$ & $4.90\times10^{-3}$ & $0.2216 ~s$& $320$ \\
                & $10$ & $4.00\times10^{-4}$ & $0.7370 ~s$ & $500$\\
                & $12$ & $2.41\times10^{-5}$ & $3.3276 ~s$ & $620$\\
                \hline
 \multirow{3}{*}{HSLRTA}
                & $12$ & $1.04\times10^{-4}$ & $0.4719 ~s$& $1500$ \\
                & $14$ & $6.30\times10^{-5}$ & $0.1739 ~s$ & $1150$\\
                & $16$ & $1.08\times10^{-5}$ & $0.1223 ~s$ & $1300$\\
\Xhline{1pt}
\end{tabular}
\label{tab1-exam1}
\end{table}

Figure \ref{fig1-exam1} plots the coefficients corresponding to the Legendre polynomial basis function of the rastrigin function using OLS, OMP, and FILARS, respectively. From Figure \ref{fig1-exam1}, we find that: (1) the dominant coefficients are concentrated in the  terms of Legendre polynomials basis functions with lower degree; (2) for this example, $l_1$-norm optimization problem (\ref{sparse-optimization-1}) gives almost the same solution as  the $l_0$-norm optimization problem (\ref{sparse-optimization-0}), and the non-zero coefficients can be accurately sought out by the FILARS. For these methods, they use different numbers of samples for function evaluation to get the approximation.  The numbers of samples are listed in Table \ref{tab1-exam1}.

The relative mean errors $\varepsilon$, the number of samples $M$ for function evaluation, and the offline CPU time for different approaches (OLS, OMP, FILARS and HSLRTA) and different polynomial degrees $p$ are listed in Table \ref{tab1-exam1}.
The offline CPU time for FILARS consists of  two parts: one  is from ILARS, the other is from finding  optimal ILARS solution using fast leave-one-out method. The relative mean errors $\varepsilon$  is defined by
\[
\varepsilon=\frac{1}{N}\sum_{i=1}^N\frac{|w(\bm{\xi}_i)-\hat{w}(\bm{\xi}_i)|}{|w(\bm{\xi}_i)|},
\]
where $N=1000$ is the number of samples used to compute the mean error, and  $\hat{w}(\bm{\xi})$ is the approximation of $w(\bm{\xi})$ obtained by OLS, OMP, MFLARS, or HSLRTA.  From Table \ref{tab1-exam1}, we can see:  (1) as the polynomial degree $p$ increases, the relative mean error becomes smaller, and it will need more offline CPU time and more samples to get the approximation for these methods; (2) for HSLRTA method, as the polynomial degree $p$ increases, the approximation error and the offline CPU time used to construct the approximation steadily decay, and the number of samples keep relatively stable; (3) the solution of $l_0$-norm optimization problem (\ref{sparse-optimization-0}) can be approximated well by solving the $l_1$-norm optimization problem for this example.

\subsection{A multivariate function dependent on both spatial variable and random parameter}
\label{Numerical-result2}
In this section, we consider a multivariate function with both  spatial variable and random parameter to illustrate the performance of NVS. We use  FILARS to approximate $\{\zeta_i(\bm{\xi})\}_{i=1}^{N}$ and get the surrogates $\{\hat{\zeta}_i(\bm{\xi})\}_{i=1}^{N}$ for  online computation.

We consider the function defined by
\[
G(x,\bm{\xi}):=\frac{1}{\exp(x_1\xi_1/2+x_2\xi_2/2+x_1x_2\xi_3/2)},  ~~x \in D=(0,1)^2,
\]
where the random vector $ \bm{\xi}:=(\xi_1, \xi_2,\xi_3)\in \mathbb{R}^{3}$ and $\xi_i\sim N(0, 1)$  ($i=1,2,3$), i.e., normal distribution with zero mean and unit variance.
For the discretization of the spatial domain, we use $50\times50$ uniform grid.   The NVS in Algorithm \ref{algorithm-pgddf} is used to get the approximation in the variable-separation  form (\ref{separat-presentation}). In order to get the mutually independent surrogates $\{\hat{\zeta}_i(\bm{\xi})\}_{i=1}^{N}$ for $\{\zeta_i(\bm{\xi})\}_{i=1}^{N}$,  we use FILARS to obtain an accurate approximation of $\zeta_i(\bm{\xi})$ for each $i=1,\cdots,N$. To approximate the random parameter space, we use Hermite polynomial basis functions with total degree up to $N_g=9$, thus the number of the total basis functions is $220$.  With regard to FILARS, we take $1000$ samples to get approximation for $\zeta_i(\bm{\xi})$ ($i=1,\cdots,N$).

\begin{figure}[htbp]
\centering
  \includegraphics[width=4in, height=2.5in]{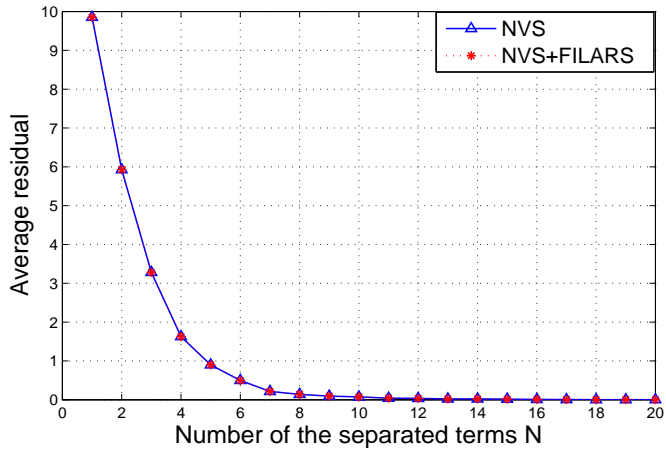}
  \caption{Comparison of the average residual corresponding to the different numbers of the separated terms $N$ for NVS method and NVS with FILARS method.}
  \label{fig2-exam1}
\end{figure}

\begin{figure}[htbp]
\centering
  \includegraphics[width=4in, height=2.5in]{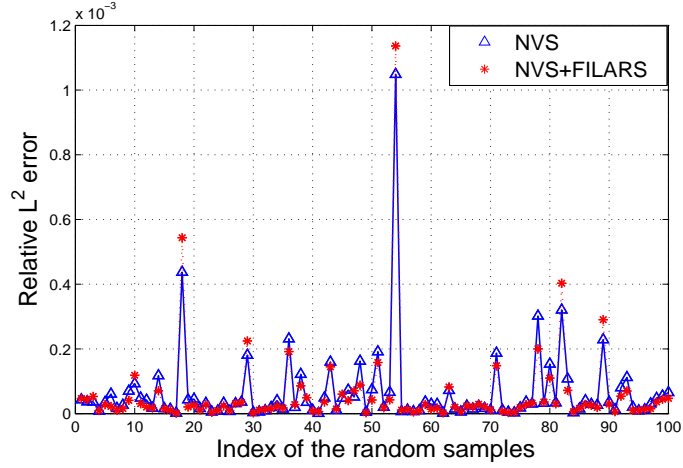}
  \caption{The relative $L^2$-error by NVS method and NVS with FILARS method with the number of the separated terms $N$ being 20.}
  \label{fig3-exam1}
\end{figure}

\begin{figure}[htbp]
\centering
  \includegraphics[width=3in, height=2.2in]{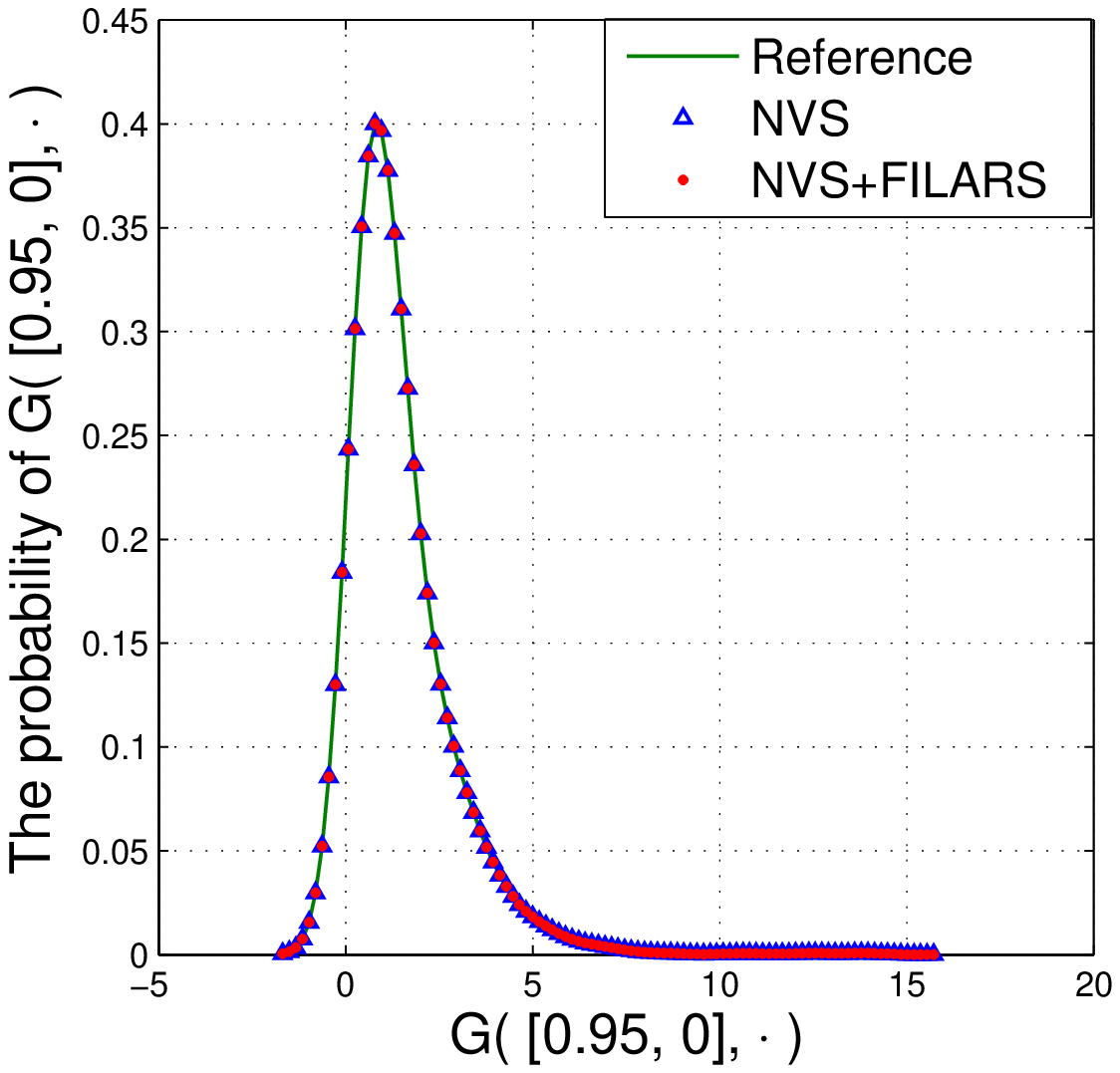}
  \includegraphics[width=3in, height=2.2in]{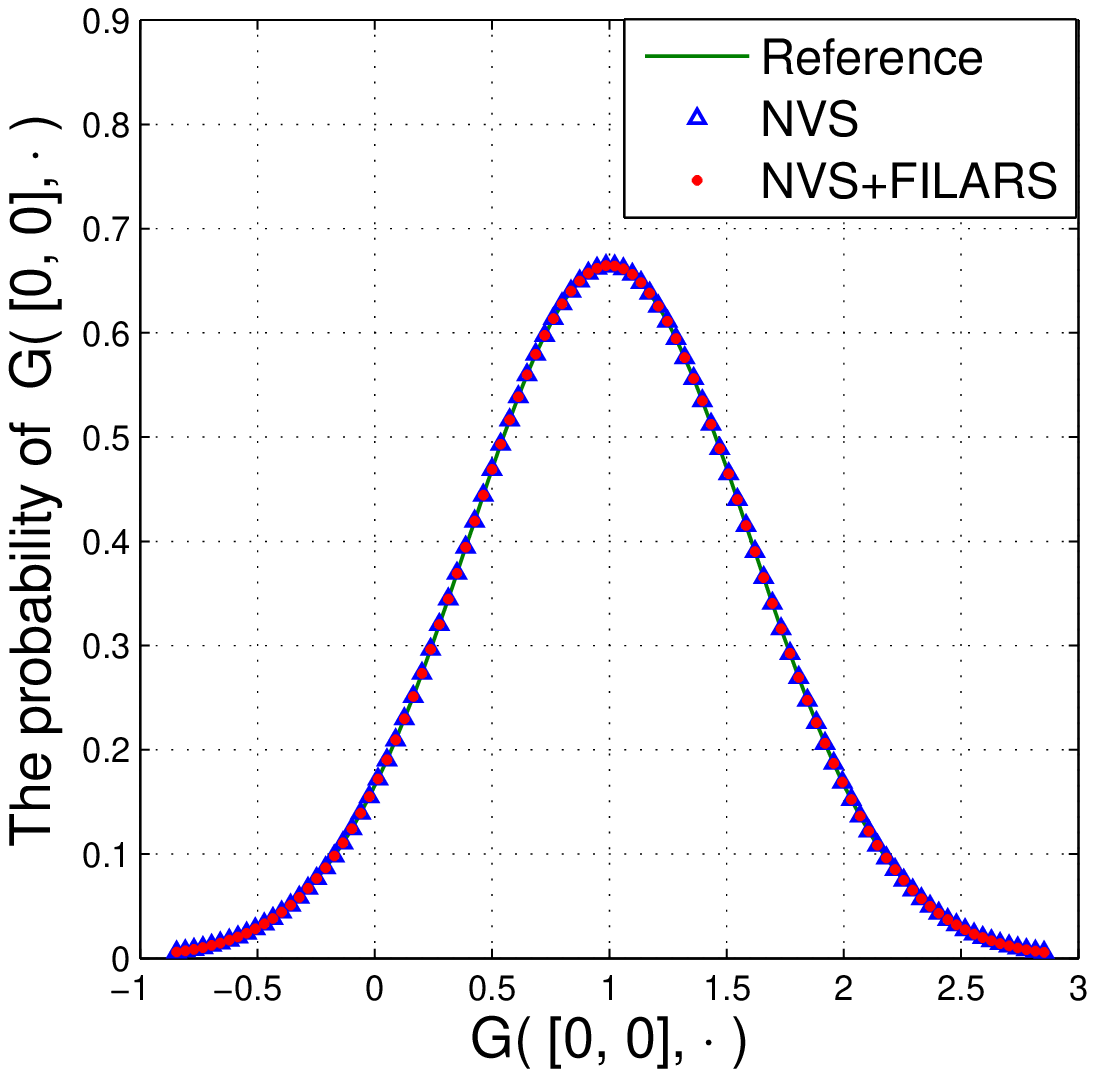}
  \caption{Probability density of the scalar $G(x_0,\bm{\xi})$ for the reference, NVS and NVS with FILARS method, where the variance of $G(x_0,\bm{\xi})$ is maximal (left) or minimal (right) for all $x \in D$, the number of the separated terms $N=20$.}
  \label{fig4-exam1}
\end{figure}

\begin{table}[hbtp]
\small
\centering
\caption{ Comparison of relative mean errors $\varepsilon$ for different approaches with the number of the separated terms $N$ being 20.}
\begin{tabular}{c|c|c}
\Xhline{1pt}
  Strategies & NVS+ILARS & NVS \\
  \hline
 $\varepsilon$ & $8.68\times10^{-5}$ & $7.28\times10^{-5} $ \\
 \hline
 Online CPU time per sample & $0.0021 s$ & $0.3545 s$ \\
\Xhline{1pt}
\end{tabular}
\label{tab2-exam1}
\end{table}

 Based on these representations, $10000$ samples $\{\bm{\xi}^{(i)}\}_{i=1}^{10000}$ are chosen to compute the average relative error. We define the average relative error as follows,
 \begin{eqnarray*}
\label{errors-u}
\varepsilon=\frac{1}{N}\sum_{i=1}^N\frac{\|G(x,\bm{\xi}^{(i)})-\tilde{G}(x,\bm{\xi}^{(i)})\|_{L^2}}{\|G(x,\bm{\xi}^{(i)})\|_{L^2}},
\end{eqnarray*}
where $N=10000$ and $\tilde{G}(x,\bm{\xi}^{(i)})$ is the approximation by NVS or NVS with FILARS. Let $r_k$ be defined by (\ref{pgd-residual}).
 In Figure \ref{fig2-exam1}, we depict the average residual $\|r_k\|_{\mathcal{V} \otimes \mathcal{S}}^2$ versus number of the separated terms $N$ for NVS method and NVS with FILARS method based on $500$ random samples. By the figure we have two observations: (1) the average residual becomes  smaller when the number of separated terms $N$ increase for the both two methods; (2) the curves of the average residual are nearly identical for the two methods. This implies that FILARS can provide an accurate approximation for random functions $\{\hat{\zeta}_i(\bm{\xi})\}_{i=1}^{N}$.

We list the average relative errors and the average online CPU time in Table \ref{tab2-exam1}. From the table, we can see: (1) Both NVS and NVS with FILARS can achieve a good approximation under the separated form; (2) the average CPU time per sample by NVS with FILARS is much smaller than that by NVS, and the approximation obtained by FILRS achieves a good trade-off in both approximation accuracy.

We plot the relative errors for the two methods in Figure \ref{fig3-exam1} to visualize the individual relative errors of the first $100$ samples. By the figure, we can see that both the two methods give good approximations.

Based on the $10000$ random samples, the probability density estimates of $G(x,\bm{\xi})$ at a single measurement location are shown in Figure \ref{fig4-exam1}. From the figure, we find that both NVS and NVS with FILARS can give the good approximations for the probability density.

\subsection{An elliptic PDE with high-dimensional random variables}
\label{Numerical-result3}

Let $k(x, \bm{\xi}): D\times \Omega \longrightarrow \mathbb{R}$ be a diffusion coefficient function.  We consider the following model elliptic equation for numerical computation,
\begin{eqnarray}
\label{exp2}
\begin{cases}
\begin{split}
-\text{div}\big(k(x,\bm{\xi})\nabla u(x, \bm{\xi})\big)& = f(x, \bm{\xi}) \quad \text{in} \quad D\times \Omega,\\
u(x,\bm{\xi})& = 0 \quad  \text{on  }  \Gamma_1,\\
k(x,\bm{\xi})\nabla u(x, \bm{\xi})\cdot \mathbf{n} & = 0 \quad \text{on  other  boundaries},
\end{split}
\end{cases}
\end{eqnarray}
where the source term $f(x, \bm{\xi})$ is defined by
\[
f(x,\bm{\xi})=2\exp(x_1+x_2+3)\sin(\xi_1\xi_{32}).
\]
Here the physical domain $D=(0,1)^2$, Dirichlet boundary  $\Gamma_1=(0, 1)\times1$, and $\mathbf{n}$ denotes the outward unit normal vector on $\partial D\setminus \Gamma_1$.
The diffusion coefficient $k(x, \bm{\xi})$ is a random field, which is characterized by  a two point exponential  covariance function $\text{cov}[k]$,  i.e.,

\begin{eqnarray*}
\label{gRB basis}
\begin{split}
\text{cov}[k](x_1, y_1; x_2, y_2)= \sigma^2 \exp\big(-\frac{|x_1-x_2|^2}{2l_x^2}-\frac{|y_1-y_2|^2}{2l_y^2}\big),
\end{split}
\end{eqnarray*}
where $(x_i, y_i)$ ($i=1,2$) is the spatial coordinate in $D$. Here the variance $\sigma^2=3$, correlation length $l_x=l_y=0.5$. The random coefficient  $k(x, \bm{\xi})$ is obtained by truncated by a Karhunen-Lo\`{e}ve expansion, i.e.,
\begin{eqnarray}
\label{KLE-TRUNC}
\begin{split}
k(x,\bm{\xi}):= E[k]+ \sum_{i=1}^{32} \sqrt{\gamma_i}b_i(x) \xi_i.
\end{split}
\end{eqnarray}
 Here $E[k]=8$ and the random vector $ \bm{\xi}:=(\xi_1, \xi_2, ...,\xi_{32})\in \mathbb{R}^{32}$. Each $\xi_i$  ($i=1,\cdots, 32$) is uniformly distributed in the interval $[-1,1]$.
For the partition of spatial domain, $100\times 100$ grid is used to compute the reference solution and solve equation (\ref{SPDE-strong-st-red}) to get $\{g_i(x)\}_{i=1}^N$. Hence, the degree of freedom  is $N_f=6241$  for FEM. We apply the novel variable-separation (NVS) method to get the variable separation representation of the solution for the elliptic PDE (\ref{exp2}). Then $\{\zeta_i(\bm{\xi}), g_i(x)\}_{i=1}^N$ are obtained by Algorithm \ref{algorithm-pgddf}, we note that $\{\zeta_i(\bm{\xi})\}_{i=1}^N$ are determined by equation (\ref{exi}). Since $\zeta_k(\bm{\xi})$ depends on $\{\zeta_i(\bm{\xi})\}_{i=1}^{k-1}$, which will impact on the computation efficiency, and bring great challenge for numerical simulation as $N$, i.e., the number of terms for (\ref{separat-spde}) increase. To overcome the difficulty, we  apply FILARS and HSLRTA to construct the surrogates
$\{\hat{\zeta}_i(\bm{\xi})\}_{i=1}^{N}$ for $\{\zeta_i(\bm{\xi})\}_{i=1}^{N}$, where $\{\hat{\zeta}_i(\bm{\xi})\}_{i=1}^{N}$ are mutually independent.  In order to reduce the high dimensionality difficulty of the random parameter, we use HSLRTA to decompose a high-dimensional problem  into a few low-dimensional problems.     To this end,
we split $\bm{\xi}=(\xi_1,\cdots, \xi_{32})$ into $4$ mutually independent sets  $\{\bm{\xi}_i=(\xi_{8i-7}, \cdots,  \xi_{8i})\}_{i=1}^4$ of random variables. Thus, the largest level number is $3$ according to Section \ref{HSLRTA}.  To approximate the random parameter space, we use Legendre polynomial basis functions with total degree up to $N_g=5$. If the $l_1$-norm optimization problem (\ref{sparse-optimization-1}) is solved by OLS or FILARS directly, the the total basis functions is up to $435897$. To circumvent the issue caused by the large number of basis functions, we use HSLRTA and solve the $l_1$-norm optimization problem (\ref{sparse-optimization-1}) through  solving some sub-optimization problems with only $8$  random variables.

\begin{table}[hbtp]
\small
\centering
\caption{ Comparison of the average residual corresponding to the different numbers of the separated terms $N$ for NVS method and NVS with HSLRTA method.}
\begin{tabular}{c|c|c}
\Xhline{1pt}
 Number of the separated terms $N$ & NVS+HSLRTA & NVS\\
  \hline
 $N=1$ & $67.7055$ & $67.7055$ \\
 \hline
 $N=2$ & $0.1662 $ & $0.1660$ \\
 \hline
 $N=3$ & $0.0420$ & $0.0433$ \\
 \hline
 $N=4$ & $0.0514$ & $0.0481$ \\
 \hline
 $N=5$ &  $0.0056$ & $0.0054$ \\
\Xhline{1pt}
\end{tabular}
\label{tab1-exam2}
\end{table}

\begin{table}[hbtp]
\small
\centering
\caption{ Comparison of relative mean errors $\varepsilon$ for different approaches with the number of the separated terms $N$ being 5.}
\begin{tabular}{c|c|c|c}
\Xhline{1pt}
  Strategies & FEM & NVS+HSLRTA & NVS \\
  \hline
 $\varepsilon$ & & $6.50\times10^{-3}$ & $6.40\times10^{-3}$ \\
 \hline
 online CPU time per sample & $2.8947 s$ & $0.0013 s$ & $0.3376 s$ \\
\Xhline{1pt}
\end{tabular}
\label{tab2-exam2}
\end{table}

\begin{figure}[htbp]
\centering
  \includegraphics[width=4in, height=2.5in]{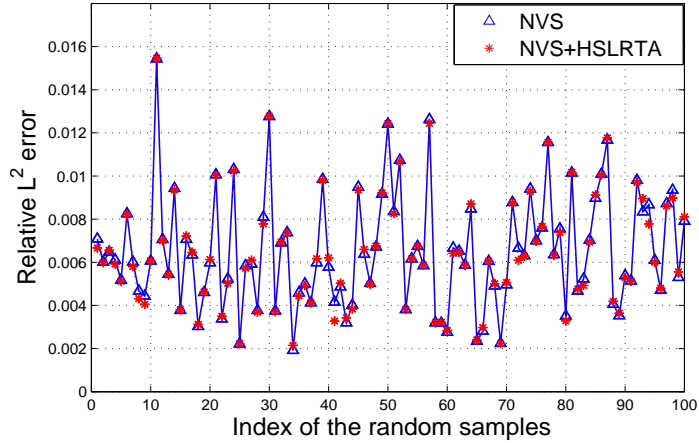}
  \caption{The relative $L^2$-error for $100$ random samples by NVS method and NVS with HSLRTA method with the number of the separate terms $N$ being 5.}
  \label{fig1-exam3}
\end{figure}

\begin{figure}[htbp]
\centering
 \includegraphics[width=2.1in, height=1.5in]{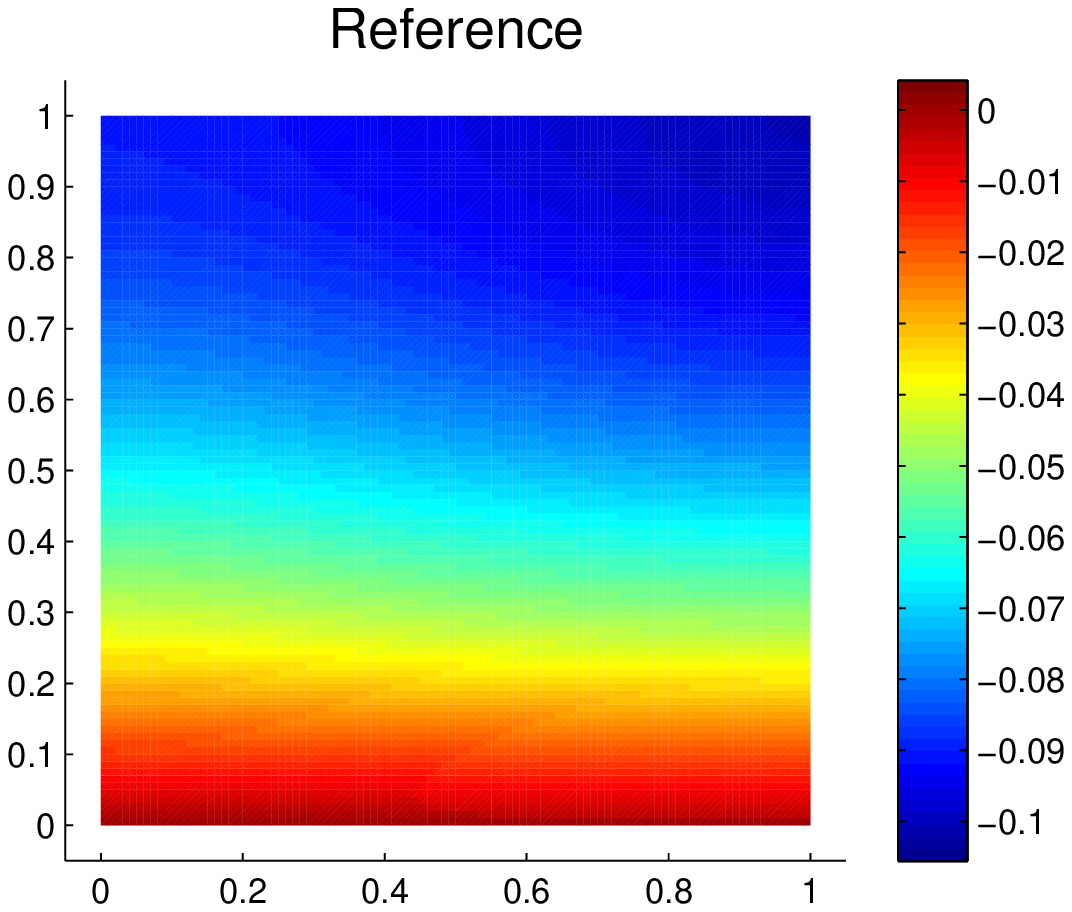}
  \includegraphics[width=2.1in, height=1.5in]{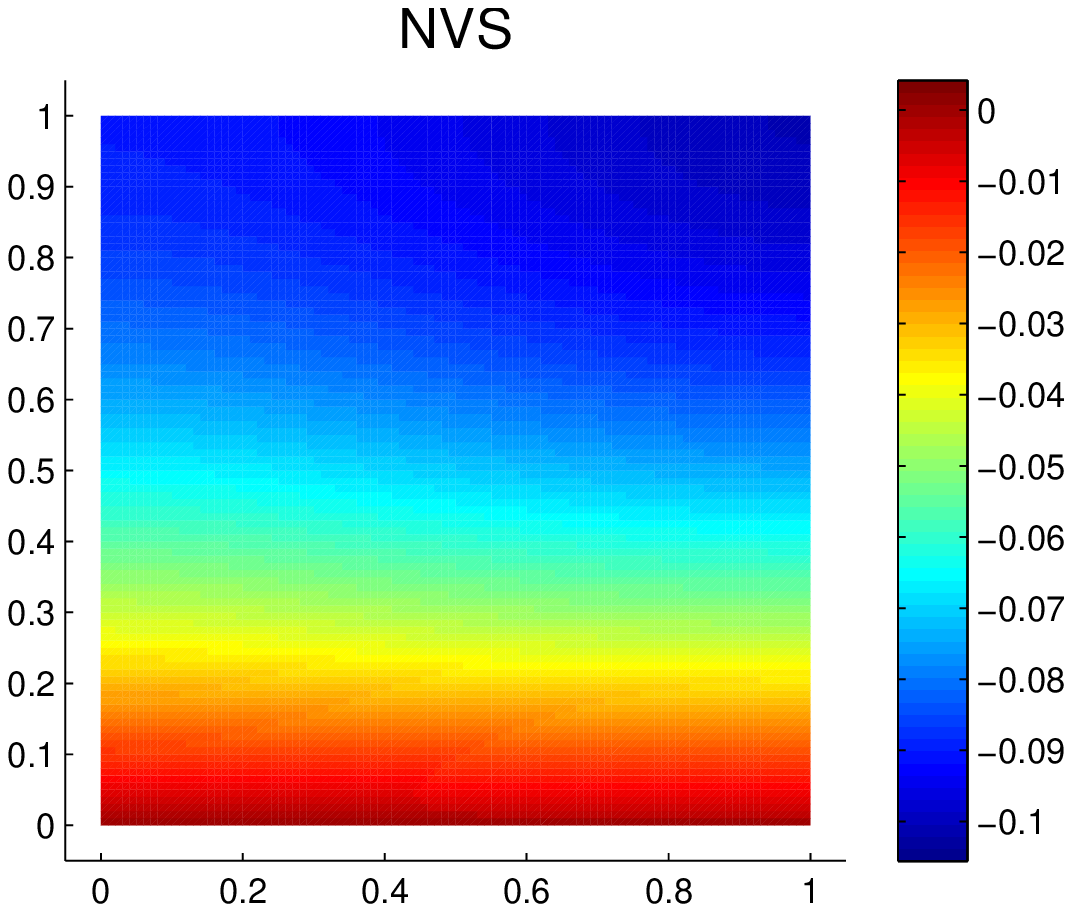}
  \includegraphics[width=2.1in, height=1.5in]{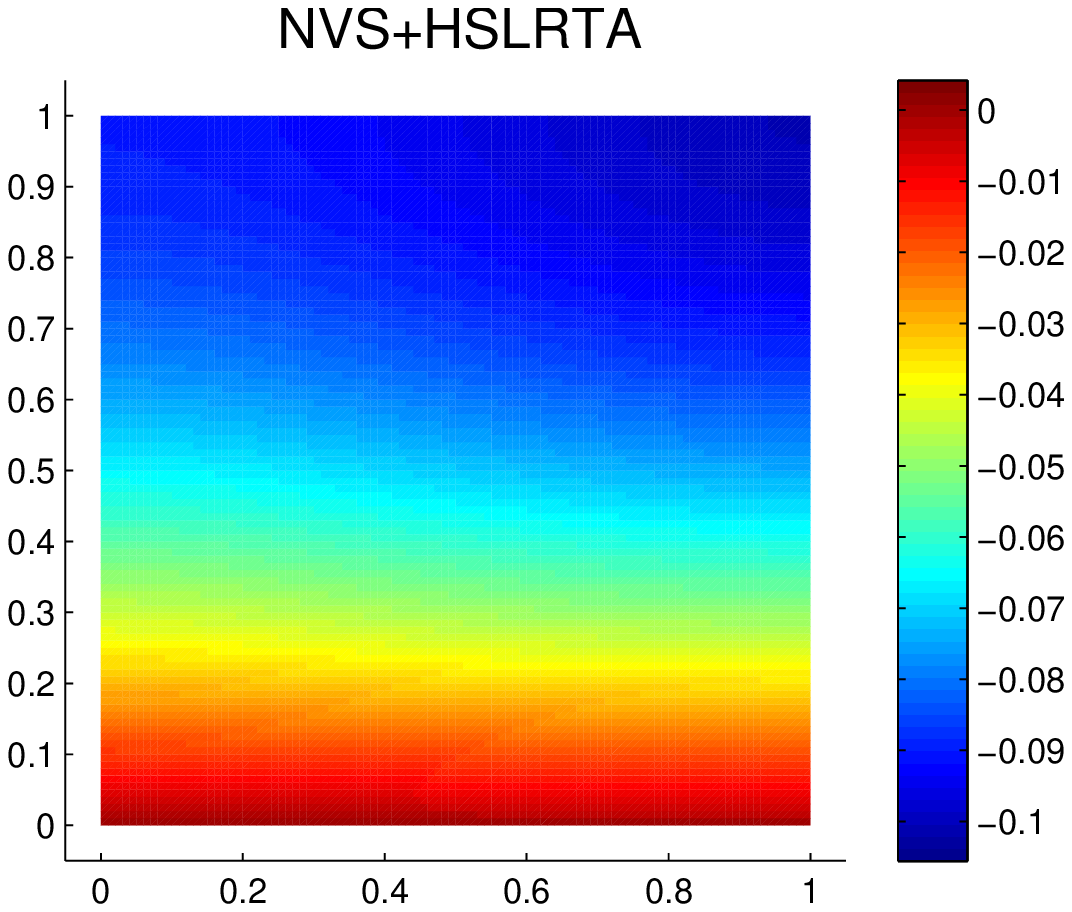}\\
  \includegraphics[width=2.1in, height=1.9in]{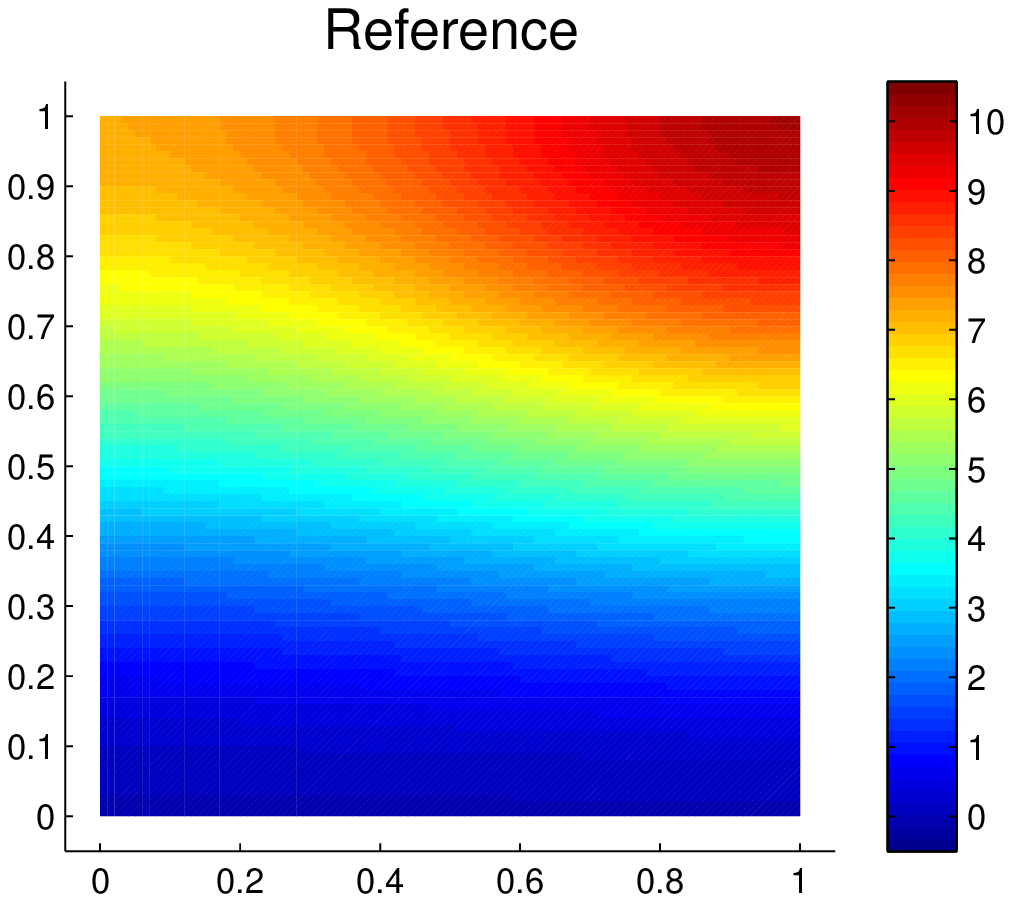}
  \includegraphics[width=2.1in, height=1.9in]{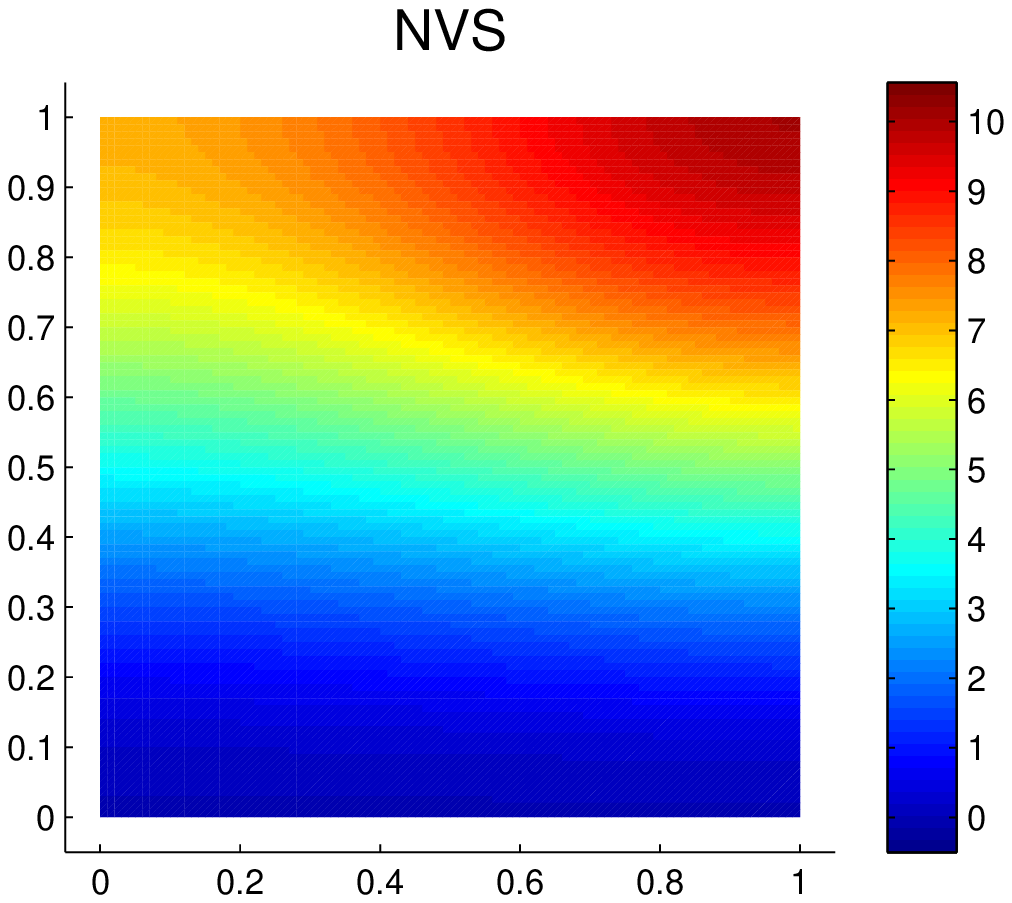}
  \includegraphics[width=2.1in, height=1.9in]{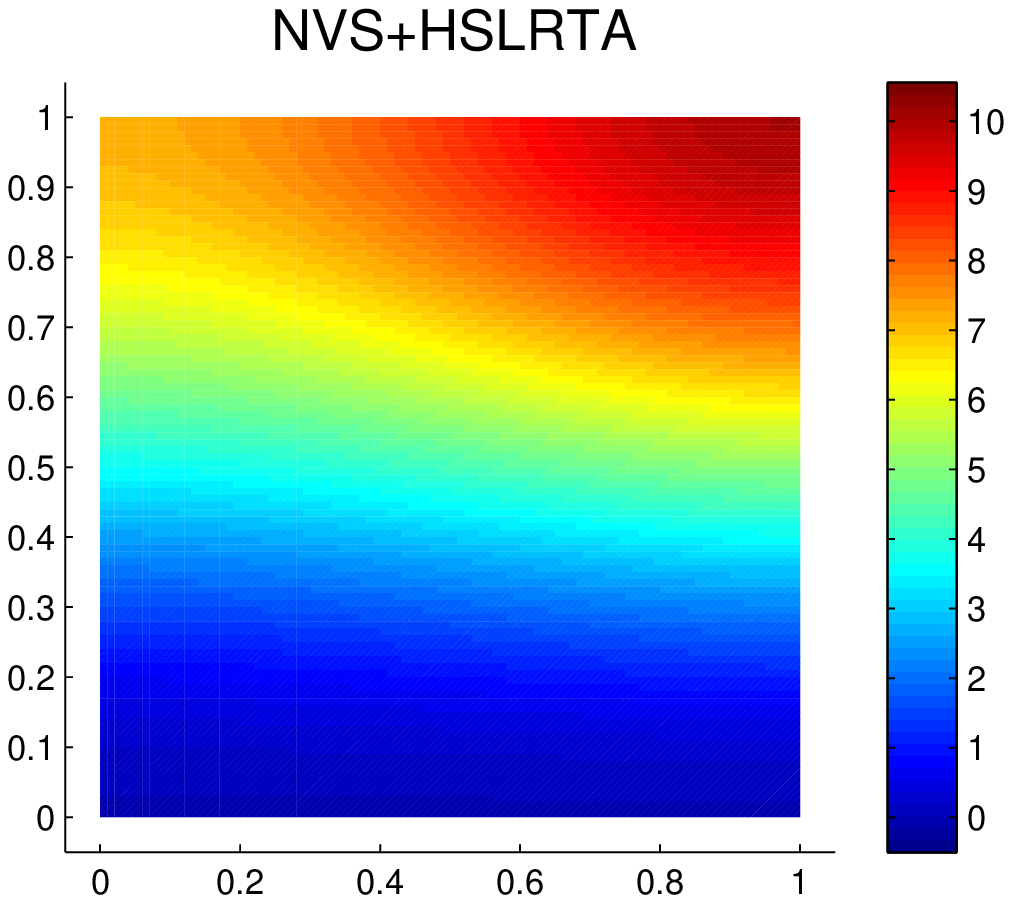}
\caption{The mean and variance of the solution $u(x,\bm{\xi})$ profiles for different methods, and the number of the separate terms is $N=5$, the first row are the mean profiles and the second row are the variance profiles.}
  \label{fig2-exam3}
\end{figure}

\begin{figure}[htbp]
\centering
  \includegraphics[width=3in, height=2.2in]{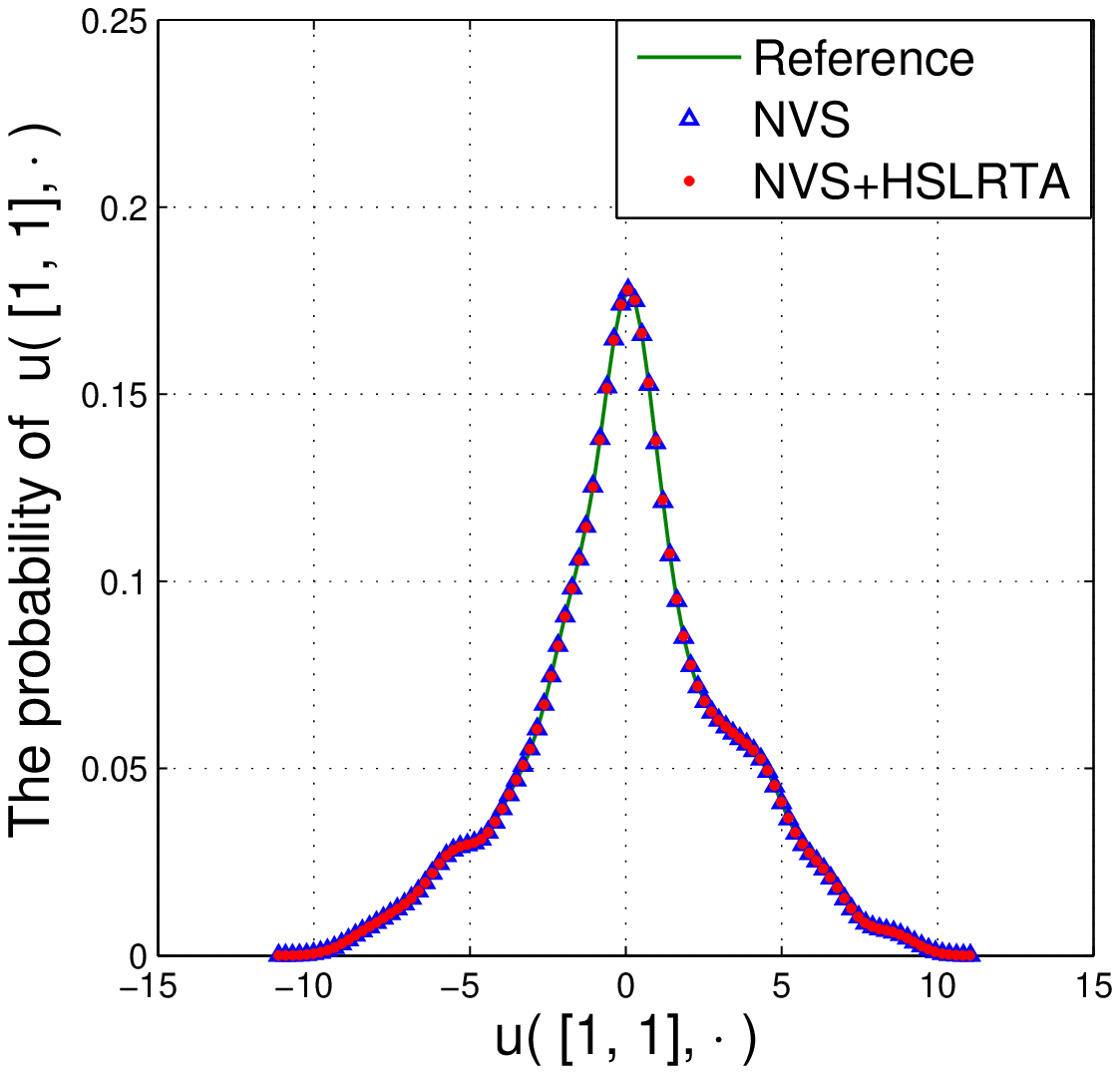}
  \includegraphics[width=3in, height=2.2in]{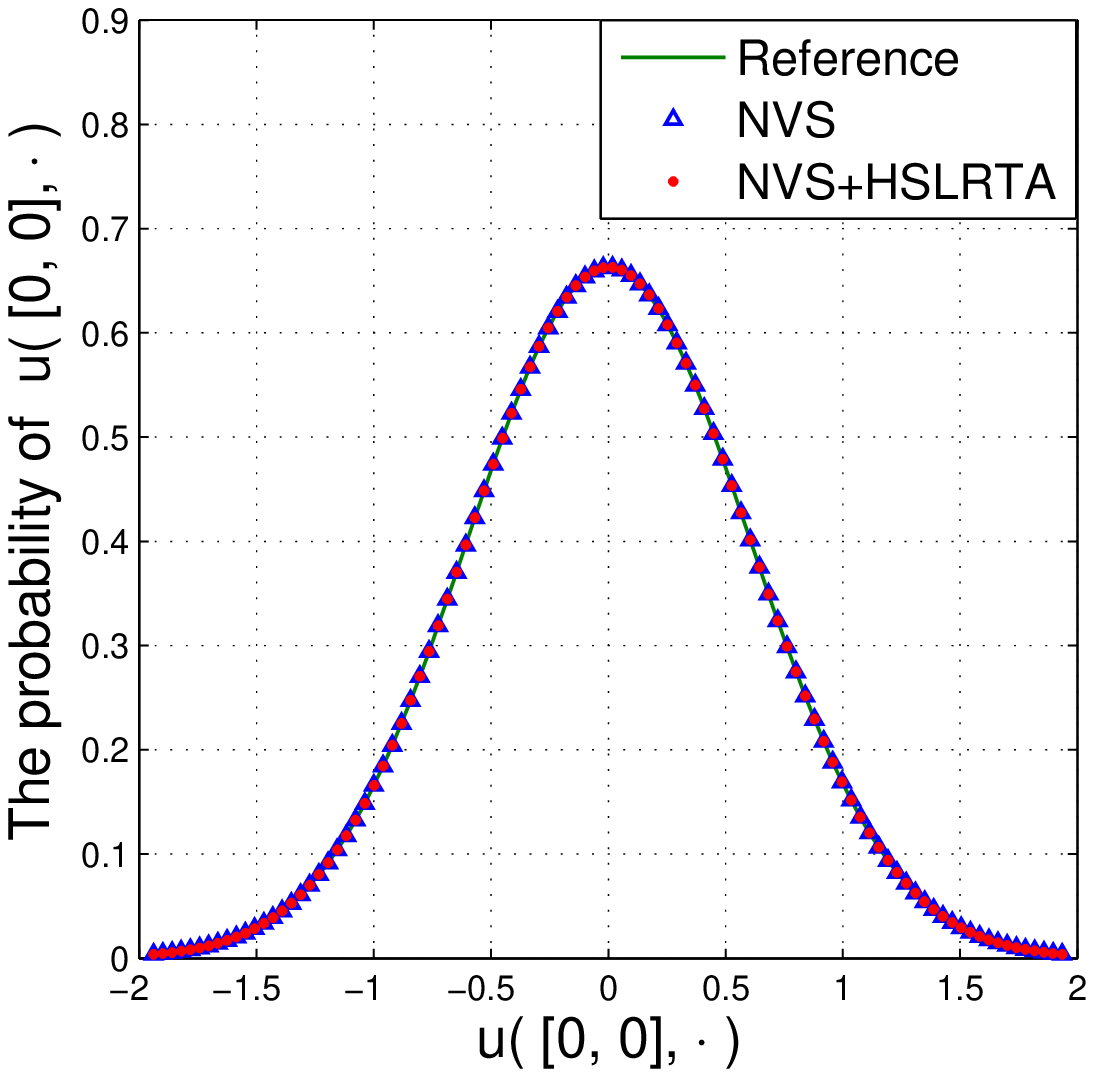}
  \caption{Probability density of  $u(x_0,\bm{\xi})$ for reference, NVS and NVS with HSLRTA, where the variance of $u(x_0,\bm{\xi})$ is maximal (left) or minimal (right) for all $x \in D$, the number of the separated terms $N=5$.}
  \label{fig3-exam3}
\end{figure}

Based on these representations, we choose $10^5$ samples $\{\bm{\xi}^{(i)}\}_{i=1}^{10^5}$ and compute the average relative error, which is defined as follows,
 \begin{eqnarray*}
\label{errors-u}
\varepsilon_u=\frac{1}{N}\sum_{i=1}^N\frac{\|u(x,\bm{\xi}^{(i)})-\tilde{u}(x,\bm{\xi}^{(i)})\|_{L^2}}{\|u(x,\bm{\xi}^{(i)})\|_{L^2}},
\end{eqnarray*}
where $N=10^5$ and $\tilde{u}(x,\bm{\xi}^{(i)})$ is the solution by NVS or NVS with HSLRTA, and $u(x,\bm{\xi}^{(i)})$ is the reference solution solved by FEM on the $100\times 100$ grid.

In Table \ref{tab1-exam2}, we depict the average residual versus number of the separated  terms $N$ for NVS method and NVS with HSLRTA method based on $500$ random samples. Here the average residual is defined by (\ref{SPDE-strong-st-red}) with $v=e(\bm{\xi})$.
By the table,  we can find that as the number of separated terms $N$ increase, the average residual becomes  smaller  for the both two methods. This implies a good approximation for random functions $\{\hat{\zeta}_i(\bm{\xi})\}_{i=1}^{N}$ using  HSLRTA. We list the average relative errors in Table \ref{tab2-exam2} along with the average online CPU time based on $100000$ random samples. From the table, we find that: (1) Both NVS and NVS with HSLRTA can provide  very good approximations; (2) the approximation obtained by HSLRTA achieves a good trade-off in both approximation accuracy and computation efficiency, the average CPU time per sample by NVS with HSLRTA is much smaller than that by NVS and FEM.

Figure \ref{fig2-exam3} demonstrates the mean and variance of solution $u(x,\bm{\xi})$ profiles for different methods. By the figure, we find that: (1) the mean profiles for the three methods are all nearly identical; (2) there is no clear difference for the variance profiles among the reference solution, NVS solution and the solution by NVS with HSLRTA.
To visualize the individual relative errors of the first $100$ samples, we plot the relative errors for the two methods in Figure \ref{fig1-exam3}, which shows that both the two methods have good agreement.

The probability density estimate of $u(x,\bm{\xi})$ based on $10^5$ random samples at a single measurement location are shown in Figure \ref{fig3-exam3}. From the figure, we can see that both NVS and NVS with HSLRTA gives  the same probability density as the reference probability density.

\section{Conclusions}
\label{ssec:Conclusions}
In the paper, we proposed a novel variable-separation (NVS) method to get a representation for multivariable functions. To achieve offline-online computation decomposition, NVS can be used to get a affine representation for model's inputs. NVS shared the merits with the EIM widely used for variable separation, but NVS is easier to implement than EIM. Firstly, the optimal parameter values and interpolation nodes are not necessary for NVS. In addition, we can compute the approximation directly by the separated representation instead of solving an algebraic system based on the optimal parameter values and interpolation nodes for the online computation, which is required for EIM.
We developed the novel variable-separation method to represent the solution in the tensor product structure for stochastic partial differential equations in high stochastic dimension. While dealing with the SPDEs in  high stochastic dimension spaces, NVS can circumvent the curse of dimensionality,  which results in the dramatic increase in the dimension of stochastic approximation spaces. Compared with proper generalized decomposition (PGD), NVS has no need to perform the suitable iterative scheme to compute $\zeta_i(\bm{\xi})$ and $g_i(x)$ at each enrichment step $i$. Since The mutual dependance of the stochastic functions $\{\zeta_i(\bm{\xi})\}_{i=1}^{N}$  would bring great challenge for numerical simulation especially when the number of terms $N$ for (\ref{separat-spde}) is large. We developed improved least angle regression algorithm (ILARS) and hierarchical sparse low rank tensor approximation method (HSLRTA) based on parse regularization to get the approximation $\hat{\zeta}_k(\bm{\xi})$ of $\zeta_k(\bm{\xi})$ such that  $\{\hat{\zeta}_i(\bm{\xi})\}_{i=1}^{N}$ are mutually independent. For ILARS, we gave the selection of the optimal regularization parameter at each step based on least angle regression algorithm (LARS) for lasso problems. This significantly improved the efficiency of   ILARS. HSLRTA was proposed to  construct an accurate approximation for high dimensional stochastic problems. We applied the proposed methods to a few numerical models with  random inputs. Careful numerical analysis was carried out for these numerical examples. In the future, we will apply the proposed methods to the nonlinear models and  explore rigorous convergence analysis.



\begin{thebibliography}{99}


\bibitem{babuvska2007stochastic}
{\sc I. Babu{\v{s}}ka, F. Nobile and R. Tempone}, {\em A stochastic collocation method for elliptic partial differential equations with random input data}, SIAM Journal on Numerical Analysis, 45 (2007), pp. 1005--1034.

\bibitem{babuvska2005solving}
{\sc I. Babu{\v{s}}ka, R. Tempone and G. Zouraris}, {\em Solving elliptic boundary value problems with uncertain coefficients by the finite element method: the stochastic formulation}, Computer methods in applied mechanics and engineering, 194 (2005), pp. 1251--1294.

\bibitem{blatman2008sparse}
{\sc G. Blatman, and B. Sudret}, {\em Sparse polynomial chaos expansions and adaptive stochastic finite elements using a regression approach}, Comptes Rendus M{\'e}canique, 336 (2008), pp. 518--523.

\bibitem{Buffa2012priori}
{\sc  A. Buffa, Y. Maday, A. Patera, C. Prud'homme and G. Turinici}, {\em A priori convergence of the greedy algorithm for the parametrized reduced basis method},  ESAIM: Mathematical Modelling and Numerical Analysis. 46 (2012), pp. 595-603.

\bibitem{chevreuil2015least}
{\sc M. Chevreuil, R. Lebrun, A. Nouy, and  P. Rai}, {\em A least-squares method for sparse low rank approximation of multivariate functions}, SIAM/ASA Journal on Uncertainty Quantification, 3 (2015), pp. 897--921.

\bibitem{doostan2009least}
{\sc A. Doostan and G. Iaccarino}, {\em A least-squares approximation of partial differential equations with high-dimensional random inputs}, J. Comput. Phys, 228 (2009), pp. 4332--4345.

\bibitem{elad2009sparse}
{\sc M. Elad}, {\em Sparse and redundant representations: from theory to applications in signal and image processing}, Springer, 2009.

\bibitem{falco2012proper}
{\sc A. Falc{\'o}, and A. Nouy}, {\em  Proper generalized decomposition for nonlinear convex problems in tensor Banach spaces},  Numerische Mathematik. 121 (2012),  pp. 503-530.

\bibitem{frauenfelder2005finite}
{\sc P. Frauenfelder, C. Schwab, and R. Todor}, {\em Finite elements for elliptic problems with stochastic coefficients}, Computer methods in applied mechanics and engineering, 194 (2005), pp. 205--228.

\bibitem{ganapathysubramanian2007sparse}
{\sc B. Ganapathysubramanian, and N. Zabaras}, {\em Sparse grid collocation schemes for stochastic natural convection problems}, Journal of Computational Physics, 225 (2007), pp. 652--685.

\bibitem{ghanem1999ingredients}
{\sc R. Ghanem}, {\em Ingredients for a general purpose stochastic finite elements implementation},  Comput Methods Appl Mech Eng, 168  (1999), ~pp. 19--34.

\bibitem{ghiocel2012stochastic}
{\sc D.M. Ghiocel, R.G. Ghanem}, {\em Stochastic finite-element analysis of seismic soil-structure interaction}, Journal of Engineering Mechanics, 128 (2002), pp. 66--77.

\bibitem{grasedyck2013literature}
{\sc L. Grasedyck, D. Kressner, and C. Tobler}, {\em A literature survey of low-rank tensor approximation techniques}, GAMM-Mitteilungen, 36 (2013),  pp. 53-78.

\bibitem{hackbusch2012tensor}
{\sc W. Hackbusch}, {\em Tensor spaces and numerical tensor calculus}, Springer, 2012.

\bibitem{jl2016}
{\sc L. Jiang and Q. Li}, {\em Model's sparse representation based on reduced  mixed GMsFE basis methods}, arxiv preprint (http://arxiv.org/abs/1605.02840), 2016.

\bibitem{keese2003review}
{\sc A. Keese}, {\em A review of recent developments in the numerical solution of stochastic partial differential equations (stochastic finite elements)}, Scientific Computing, 6 (2003).

\bibitem{keese2004adaptivity}
{\sc A. Keese and H. Mathhies}, {\em Adaptivity and sensitivity for stochastic problems}, Computational stochastic mechanics, 4 (2004), pp. 1--311.
\bibitem{kolda2009tensor}
{\sc T. Kolda, and B. Bader}, {\em Tensor decompositions and applications}, SIAM, 51 (2009), pp. 455--500.

\bibitem{khoromskij2012tensors}
{\sc B. Khoromskij}, {\em  Tensors-structured numerical methods in scientific computing: Survey on recent advances}, Chemometrics and Intelligent Laboratory Systems, 110 (2012), pp. 1--19.

\bibitem{khoromskij2011tensor}
{\sc B. Khoromskij, and C. Schwab}, {\em Tensor-structured Galerkin approximation of parametric and stochastic elliptic PDEs},  SIAM Journal on Scientific Computing, 33 (2011),  pp. 364-385.

\bibitem{Kolmogoroff1936}
{\sc A. Kolmogoroff}, {\em $\ddot{U}$ber die beste Ann$\ddot{a}$herung von Funktionen einer gegebenen Funktionenklasse},  Ann. of Math, 37 (1936),  pp. 107-110.


\bibitem{matthies2008stochastic}
{\sc H. Matthies}, {\em Stochastic finite elements: Computational approaches to stochastic partial differential equations}, Z Angew
Math Mech, 88 (2008),  pp. 849-873.

\bibitem{matthies2005galerkin}
{\sc H. Matthies, and A. Keese}, {\em Galerkin methods for linear and nonlinear elliptic stochastic partial differential equations}, Computer Methods in Applied Mechanics and Engineering. 194 (2005), pp. 1295-1331.

\bibitem{mathelin2007dual}
{\sc L. Mathhies, O. Le Ma{\^\i}tre}, {\em Dual-based a posteriori error estimate for stochastic finite element methods}, Commun Appl Math
Comput Sci, 2 (2007), pp. 83--115.

\bibitem{Migliorati2014Analysis}
{\sc  G. Migliorati, F. Nobile, E. von Schwerin and R. Tempone}, {\em Analysis of discrete $L^2$ projection on polynomial spaces with random evaluations}, Found Comput Math. 14 (2014), pp. 419-456.

\bibitem{nouy2007generalized}
{\sc A. Nouy}, {\em A generalized spectral decomposition technique to solve a class of linear stochastic partial differential equations}, Comput Methods Appl Mech Eng, 196 (2007), pp. 4521--4537.

\bibitem{nouy2009recent}
{\sc A. Nouy}, {\em Recent developments in spectral stochastic methods for the numerical solution of stochastic partial differential equations}, Archives of Computational Methods in Engineering, 16 (2009), pp. 251-285.

\bibitem{nouy2009generalized}
{\sc A. Nouy and O. Le Ma{\^\i}tre}, {\em Generalized spectral decomposition for stochastic nonlinear problems}, Journal of Computational Physics, 228 (2009), pp. 202--235.

\bibitem{nouy2010proper}
{\sc A. Nouy}, {\em Proper generalized decompositions and separated representations for the numerical solution of high dimensional stochastic problems}, Arch Comput Methods Eng, 17 (2010), pp. 403--434.

\bibitem{nobile2008sparse}
{\sc F. Nobile, R. Tempone, C. Webster}, {\em A sparse grid stochastic collocation method for partial differential equations with random input data}, SIAM Journal on Numerical Analysis, 46 (2008), pp. 2309--2345.

\bibitem{prv02}
{\sc C. Prud'homme, D. Rovas, K. Veroy, Y. Maday, A. Patera, and G. Turinici}, {\em Reliable real-time solution of parametrized  partial differential equations: reduced-basis output bounds methods}, Journal of Fluids Engineering, 124 (2002), pp. 70-80.

\bibitem{rish2014sparse}
{\sc I. Rish, and G. Grabarnik}, {\em Sparse modeling: theory, algorithms, and applications}, CRC Press, 2014.

\bibitem{rhp08}
{\sc  G. Rozza, D.B.P. Huynh, and A.T. Patera}, {\em Reduced basis approximation and a posteriori error estimation for affinely parametrized elliptic coercive partial differential equations}, Arch Comput Methods Eng. 15 (2008), pp. 229-275.

\bibitem{tropp2007signal}
{\sc J. Tropp and A. Gilbert}, {\em  Signal recovery from random measurements via orthogonal matching pursuit}, Information Theory, IEEE Transactions on. 53 (2007), pp. 4655--4666.

\bibitem{wan2005adaptive}
{\sc X. Wan, and G. Karniadakis}, {\em An adaptive multi-element generalized polynomial chaos method for stochastic differential equations}, Journal of Computational Physics, 209 (2005), pp. 617--642.

\bibitem{wan2009error}
{\sc X. Wan, and G. Karniadakis}, {\em Error control in multi-element generalized polynomial chaos method for elliptic problems with random coefficients}, Communications in Computational physics, 5 (2009), pp. 793--820.

\bibitem{xiu2005high}
{\sc D. Xiu and J. Hesthaven}, {\em High-order collocation methods for differential equations with random inputs}, SIAM Journal on Scientific Computing, 27 (2005), pp. 1118--1139.


\bibitem{beylkin2005algorithms}
{\sc D. Xiu}, {\em  Fast numerical methods for stochastic computations: a review}, Communications in computational physics, 5 (2009), pp. 242--272.

\bibitem{xiu2007efficient}
{\sc D. Xiu}, {\em Efficient collocational approach for parametric uncertainty analysis}, Commun. Comput. Phys, 2 (2007), pp. 293--309.
\end{thebibliography}
\end{document}